\documentclass[oneside,11pt,reqno]{amsart}
\usepackage{amssymb,amsmath,amsthm,bbm,enumerate,mdwlist,url,multirow,amsthm}
\usepackage[shortlabels]{enumitem}

\usepackage{xcolor}
\usepackage[numbers]{natbib}

\usepackage{hyperref}

\theoremstyle{definition}
\newtheorem{definition}{Definition}%Extra square-bracket argument achives that the numbering is the same as for definition (single uniform counter). 
\theoremstyle{theorem}
\newtheorem{proposition}[definition]{Proposition}
\newtheorem{lemma}[definition]{Lemma}
\newtheorem{theorem}[definition]{Theorem}
\newtheorem{meta-corollary}[definition]{Meta-corollary}
\newtheorem{corollary}[definition]{Corollary}

\numberwithin{equation}{section}
\numberwithin{definition}{section}
\theoremstyle{remark}
\newtheorem{remark}[definition]{Remark}
\newtheorem{example}[definition]{Example}
\newtheorem{question}[definition]{Question}
\def\PP{\mathsf{P}}
\def\EE{\mathsf{E}}
\def\QQ{\mathsf{Q}}
\def\MM{\mathcal{M}}
\def\GG{\mathcal{G}}
\def\AA{\mathcal{A}}
\def\BB{\mathcal{B}}

\def\QQ{\mathsf{Q}}

\def\FF{\mathcal{F}}
\def\UU{\mathcal{U}}

\def\NN{\mathbb{N}}
\def\NN0{\mathbb{N}_0}
\def\id{\mathrm{id}}

\def\supp{\mathrm{supp}}
\def\conv{\mathrm{conv}}
\def\JJ{\mathfrak{I}}
\def\II{\mathcal{I}}
\def\LL{\mathcal{L}}
\def\CC{\mathcal{C}}

\begin{document}

\title{The structure of non-linear martingale optimal transport problems}
\author{Alexander M.~G.~ Cox}
\address{Department of Mathematical Sciences, University of Bath, U.K.}
\email{a.m.g.cox@bath.ac.uk}

\author{Matija Vidmar}
\address{Department of Mathematics, University of Ljubljana, and Institute of Mathematics, Physics and Mechanics, Slovenia}
\email{matija.vidmar@fmf.uni-lj.si}
\begin{abstract}
We explore the structure of solutions to a family of non-linear martingale optimal transport (MOT) problems that involve conditional expectations in the objective functional. En route general results concerning optimization over (martingale) measures are proved that appear much more widely applicable. In particular the analysis leads us to introduce a notion of so-called curtain transports; in a main contribution we highlight the r\^ole that these transports play in (non-linear) MOT.
\end{abstract}

\thanks{MV acknowledges financial support from the Slovenian Research Agency (core fundings Nos. P1-0222 \& P1-0402) and is grateful for the kind hospitality of the University of Bath, where he was on a sabbatical while this research was conducted.}



\maketitle

\section{Introduction}
%\textbf{TODO}: say sth about VIX and ``non-linear'' martingale optimal transport problems.
%\cred{AC: Rewrite...}
In this paper we will be interested in describing the structure of the solution to a class of ``non-linear'' one-step one-dimensional  martingale optimal transport (MOT) problems. Informally, the non-linearity we mention will be in the objective functional $J(\nu$) --- to be optimized over the class of all martingale couplings $\nu$ of two given probabilities  on the real line, --- and it will come from an application of a (non-linear) function to a conditional expectation [w.r.t. $\nu$] of another function, before the outer unconditional expectation [again w.r.t. $\nu$] is finally taken. %All the expectations in the preceding were of course w.r.t. a generic martingale coupling (of two given marginal probabilities), over the class of which the objective functional is finally optimized. 
En route we will establish results that shed general  light on non-linear optimal martingale transport problems. In order to motivate our base class of problems, and to describe it comfortably in further detail, we agree on the following /perhaps slightly non-standard, though certainly not new/ 

\smallskip

\noindent
{\bf General notation.} We will write $\QQ[W]$ for $\EE_\QQ[W]$, $\QQ[W;A]$ for $\EE_\QQ[W\mathbbm{1}_A]$, $\QQ[W\vert \mathcal{H}]$ for $\EE_\QQ[W\vert \mathcal{H}]$, and $Z_\star \QQ$ for the law of $Z$ under $\QQ$ w.r.t. a $\sigma$-field on the codomain that will be clear from context. Further, for $\sigma$-fields $\AA$ and $\BB$, $\AA/\BB$ will denote the set of $\AA/\BB$-measurable maps; $\mathcal{B}_A$ is the Borel (under the standard topology) $\sigma$-field on $A$; $b\mathcal{B}:=\{f\in \mathcal{B}/\mathcal{B}_\mathbb{R}:f\text{ bounded}\}$. %Finally, for a numerical function $g$, $\Vert g\Vert_\infty:=\sup\{\vert g(x)\vert:x\in \mathcal{D}g\}$. 

\subsection{Motivation: valuation of VIX futures}\label{subsection:motivation}
In this section, we motivate our problem by considering a particular financial problem (following \citet{guyon}). Let $(S_1,S_2,V)$ be the canonical projections on $(0,\infty)^2\times [0,\infty)$, and let $(X_1,X_2)$ be the canonical projections on $(0,\infty)^2$. Furthermore, we let $\mu_1$ and $\mu_2$ be probability measures on $\mathcal{B}_{(0,\infty)}$ in convex order (see Definition~\ref{def:convex}) and such that $\mu_1[\vert \ln\vert]\lor \mu_2[\vert \ln\vert]<\infty$;
%$\int |\ln(x)| \, \di \mu_i(x) < \infty$;
 we also fix  $\tau\in (0,\infty)$. 

Denote by $\mathcal{M}'$ the set of probability measures $\mu$ on  $\mathcal{B}_{(0,\infty)^2 \times [0,\infty)}$ satisfying: 
$${S_1}_\star\mu=\mu_1,\, {S_2}_\star\mu=\mu_2,\mu[S_2\vert S_1,V]=S_1\text{ and }\sqrt{\mu\left[-\frac{2}{\tau}\ln\left(\frac{S_2}{S_1}\right)\Big\vert S_1,V\right]}=V\text{ a.s.-}\mu;$$
and denote by $\mathcal{M}$ the set  of probability measures $\nu$ on  $\mathcal{B}_{(0,\infty)^2 }$ such that 
$${X_1}_\star\nu=\mu_1,\, {X_2}_\star\nu=\mu_2\text{ and }\nu[X_2\vert X_1]=X_1\text{ a.s.-}\nu,$$
the collection of all martingale transports of $\mu_1$ to $\mu_2$. 

In \citet{guyon} there is then considered a  ``primal $\inf$''  super-replication optimization problem  \cite[Subsection~2.2]{guyon} for the time-$0$ price of a futures contract on the S\&P 500 VIX volatility index, the superhedging portfolio consisting of calls on the S\&P~500 at times $1$ and $2$, and forward-starting log-contracts. Indeed the $S_1$, $S_2$ and $V^2$ above correspond respectively to the value of the S\&P 500 index at time $1$, at time $2$, and the time-$1$ price of the forward-starting log-contract. We refer the interested reader to \cite{guyon} for further details concerning this primal problem; the specifics are not important for the understanding of what follows. What is important for our results is that this primal problem is shown \cite[Section~4]{guyon} to have the ``dual $\sup$'' representation: 
\begin{equation}\label{VIX:1}
\sup_{\mu\in \mathcal{M}'}\mu[V].
\end{equation}
This problem naturally corresponds to the financial problem of finding the pricing measure which correctly prices all the quote options (calibration), and which maximises the VIX future price. Furthermore it is shown in \cite[(proof of) Proposition~4.10 and Lemma~3.3]{guyon} that the latter problem is equivalent to
\begin{equation}\label{VIX:2}
\sup_{\nu\in \mathcal{M}}\nu\left[\sqrt{-\frac{2}{\tau}\nu\left[\ln\left(\frac{X_2}{X_1}\right)\vert X_1\right]}\right],
\end{equation}
in the sense that: (i) the two suprema coincide; and, moreover, (ii) if $\nu\in \mathcal{M}$ attains the $\sup$ in \eqref{VIX:2}, then $(X_1,X_2,\sqrt{-\frac{2}{\tau}\nu[\ln(\frac{X_2}{X_1})\vert X_1]})_\star\nu\in \mathcal{M}'$ attains the $\sup$ in \eqref{VIX:1}, while conversely if $\mu\in \mathcal{M}'$ attains the $\sup$ in \eqref{VIX:1}, then $(S_1,S_2)_\star\mu\in \mathcal{M}$ attains the $\sup$ in \eqref{VIX:2}.

\subsection{A class of non-linear MOT problems}
\label{sec:class-non-linear}

Motivated by the above, we consider the following family of optimal martingale transport problems, whose structure generalizes that of  \eqref{VIX:2}. Let $\JJ$ be a non-empty open interval of $\mathbb{R}$, $\gamma:\JJ\to \mathbb{R}$ convex and $\phi:[0,\infty)\to \mathbb{R}$  concave. Then we have, for given probability measures $\mu_1$ and $\mu_2$ on $\mathcal{B}_{\JJ}$ of finite mean, for which $\mu_1[\gamma^+]<\infty$, $\mu_2[\gamma^+]<\infty$, and  in convex order, the optimization problem

\begin{equation}\label{mtg-problem-intro}
\sup_{\nu\in \mathcal{M}}J(\nu),\text{ where }J(\nu):=\nu[V_\nu]
\text{ with } V_\nu:=\phi(\nu[\gamma(X_2)\vert X_1]-\gamma(X_1))
\end{equation}
$$\text{ for }\nu \in \mathcal{M}:=\{\text{martingale transports of $\mu_1$ to $\mu_2$}\}.$$
Here $(X_1,X_2)$ are the canonical projections on $\JJ^2$. The problem \eqref{VIX:2} corresponds to $\JJ=(0,\infty)$, $\gamma=-\frac{2}{\tau}\ln$ and $\phi=\sqrt{\cdot}$. 

Note that \eqref{mtg-problem-intro} does not fall under the umbrella of ``classical''  optimal martingale transport because of the ``non-linearity'' introduced by the application of $\phi$ subsequent to the conditioning in the expression for $V_\nu$. Indeed, in the classical setting, $V_\nu$ in the above would simply be a suitable (sufficiently integrable) fixed gain function $c\in \mathcal{B}_{\JJ^2}/\mathcal{B}_{[-\infty,\infty]}$, and such classical, ``linear'', optimal martingale transport problems have received a substantial amount of attention in recent years, for example in \citet{griessler,juillet,penkner,cox,beiglbock2017,Campi2017,Dolinsky2014,numerical,Henry,Hobson2015} and \citet{hobson2012}. On the other hand, $J$ of \eqref{mtg-problem-intro} is a special case of a general gain transport function as introduced in \citet{gozlan}, albeit there for optimization over all (not just martingale) couplings. A class of unrelated non-linear optimal martingale transport problems is considered in \citet[Section~5]{biegelbock}, but beyond that precious little appears to be known in the non-linear setting. 

It is indeed the non-linearity  in \eqref{mtg-problem-intro}   --- over and above the obvious fact that we are optimizing over martingale couplings, and not just all couplings --- that makes the analysis of \eqref{mtg-problem-intro} more involved, but also more interesting. It emerges, moreover, that the family of problems \eqref{mtg-problem-intro} is sufficiently special as to make a relatively explicit description of optimality possible, and we provide a panorama of this in the next subsection.

\subsection{Overview of results}
Fiest, when $\mu_1$ has a finite support $\{a_1,\ldots,a_n\}$ of cardinality $n\in \mathbb{N}$, our results will show that the optimization problem introduced in the preceding subsection reduces structurally to two subproblems (Meta-corollary~\ref{remark:fundamental}).

The first of these subproblems is the determination of what we call the curtain martingale transports of $\mu_1$ to $\mu_2$, the class of which can be described simply in terms of ``forbidden overlapping transports'' (Definition~\ref{definition:curtain}), and the members of which can successfully be characterized both constructively (Proposition~\ref{proposition:curtain-transport}) as well as being precisely the solutions to a certain class of classical (as above) optimal martingale transport problems in which the gain function $c$ is of tensor product form (Corollary~\ref{corollary:curtain}). These curtain transports include the left- and right- curtain couplings of \citet{juillet} (see also \citet{Henry,juillet2016}) and are contained in the class of shadow couplings of \citet{biegelbock}. 
This first subproblem is independent of the particularities of the functions $\gamma$ and $\phi$, and serves indeed as a means to solve a much wider family of problems than the one given in \eqref{mtg-problem-intro}. 

The second subproblem is an optimization of a concave function (determined by $\phi$, $\gamma$ and $\mu_1$) over the compact convex polytope of Euclidean space, whose vertices are given in terms of the curtain martingale transports and $\gamma$. See Subsection~\ref{subsection:finite-support} for further details. 

The above reduction is made possible by the following result (Theorem~\ref{corollary:extremal-points} below), which is one of our main contributions: let $\mathcal{X}:=\{(\nu_1[\gamma],\ldots,\nu_n[\gamma]):\nu\in \mathcal{M}\}$, where for a $\nu\in \MM$, $\nu=:\sum_{i=1}^n\mu_1(\{a_i\})\delta_{a_i}\times \nu_i$. Then $\mathcal{X}=\conv(\mathcal{E})$, where $\mathcal{E}:=\{(\nu_1[f],\ldots,\nu_n[f]):\nu\text{ a curtain transport of }\mu_1\text{ to }\mu_2\}$. 

Second, while we were not able to prove an analogous decomposition when the first marginal is not finitely supported, under reasonably innocuous conditions, a continuity result, Theorem~\ref{theorem:general}, ensures, informally speaking, that the solution to \eqref{mtg-problem-intro} is well-approximated by the solution to the same problem when $\mu_1$ is replaced by a sufficiently fine finitely supported discretization of itself.

We note that  \citet{guyon} also considered the accompanying subreplication ``primal $\sup$'' problem for the price of the VIX futures, whose ``dual $\inf$'' problem corresponds to replacing $\sup$ by $\inf$ in \eqref{VIX:1}. However the latter is no longer equivalent to the analogue of \eqref{VIX:2}. While we will have occasion to say something about \eqref{mtg-problem-intro} in which $\inf$ replaces $\sup$ therein, we shall say nothing about \eqref{VIX:1} when $\inf$ replaces $\sup$. 

\subsection{Structure of the paper}
The organisation of the remainder of this paper is as follows. Section~\ref{section:general} delivers some general results in optimization over (martingale) measures. Section~\ref{section:main} considers in detail the family of problems \eqref{mtg-problem-intro}, applying to it in particular the results of Section~\ref{section:general}. More precisely: Subsection~\ref{subsection:intro} gives some general properties of the family \eqref{mtg-problem-intro}; Subsection~\ref{section:two-point} explores the case when $\mu_1$ has a two-point support (this assumption renders further simplifications possible); Subsection~\ref{subsection:finite-support} handles the case when the support of $\mu_1$ is finite; Subsection~\ref{subsection:general} provides an ``approximation'' theorem which connects the general case to the finitely-supported-first-marginal case; finally, Subsection~\ref{subsection:duality} establishes a duality result (a special case of which is the super-replication primal problem of \cite{guyon} mentioned above). %In Section~\ref{section:numerical} we describe some numerical approaches to the problem.

\section{Optimization over (martingale) measures} \label{section:general}
%Because our original motivation of, and eventual application to the VIX is to do with measures on $\mathcal{B}_{(0,\infty)}$, we will work below throughout on the positive half-axis $(0,\infty)$. The reader will note however that everything works equally well on the real-line, modulo the obvious adjustments. 
In this section we prove some key and quite general results about optimization over (martingale) measures (Propositions~\ref{proposition:general}, ~\ref{proposition:curtain-transport} and~\ref{proposition:curtains-optimal}; Theorem~\ref{corollary:extremal-points}), which will later be applied to the understanding of  \eqref{mtg-problem-intro} in Section~\ref{section:main}. We believe the mentioned results are interesting in their own right. Throughout this section we let $\JJ$ be a non-empty open interval of $\mathbb{R}$ and  denote by $X_1$ and $X_2$ the canonical projections on the first and second coordinate of $\JJ^2$. %modifications and minor adjustments that have to do with the integrability assumptions.

We will require the following notation and notions.
\begin{definition}
For a finite measure $\gamma$ on $\mathcal{B}_{\JJ}$ and $\{a,b\}\subset [0,\gamma[1]]$ with $a\leq b$, let $\gamma_a^b$ be the restriction of $\gamma$ between the quantilies $a$ and $b$; that is to say $\gamma_a^b$ is the unique measure $\nu$ on $\mathcal{B}_{\JJ}$ such that $\nu((-\infty,x]\cap \JJ)=(\gamma((-\infty,x]\cap\JJ)-a)^+\land (b-a)$ for all $x\in \JJ$. A measure $\nu$ on $\mathcal{B}_{\JJ}$ is called a (co-)connected part of $\gamma$ if $\nu=\gamma_a^b$ ($\nu=\gamma-\gamma_a^b$) for some $a\leq b$, $\{a,b\}\subset [0,\gamma[1]]$.
\end{definition}

\begin{remark}\label{remark:connected-part}
Let $\gamma$ be a finite measure on $\mathcal{B}_{\JJ}$ with finite first moment (i.e.~$\gamma[\mathrm{id}]\in \mathbb{R}$ is well-defined and finite). Then, for any $a\in [0,\gamma[1]]$, the map $[0,\gamma[1]-a]\ni c\mapsto \gamma_c^{c+a}[\mathrm{id}]$ is real-valued, nondecreasing, continuous, and its intervals of constancy coincide with those of $[0,\gamma[1]-a]\ni c\mapsto \gamma_c^{c+a}$. Therefore it maps $[0,\gamma[1]-a]$ onto $[\gamma_0^a[\mathrm{id}], \gamma_{\gamma[1]-a}^{\gamma[1]}[\mathrm{id}]]$ and, moreover, for any $b\in [\gamma_0^a[\mathrm{id}], \gamma_{\gamma[1]-a}^{\gamma[1]}[\mathrm{id}]]$ there is a unique connected part of $\gamma$ of mass $a$ and first moment $b$. Correspondingly, again for any $a\in [0,\gamma[1]]$ and $b\in  [\gamma_0^a[\mathrm{id}], \gamma_{\gamma[1]-a}^{\gamma[1]}[\mathrm{id}]]$, there is also  a unique co-connected part of $\gamma$ of mass $a$ and first moment $b$.
\end{remark}

\begin{proposition}\label{proposition:general}
Let $\gamma$ be a finite measure on $\mathcal{B}_{\JJ}$ of finite first moment. Let $a\in [0,\gamma[1]]$ and $b\in [\gamma^{a}_{0}[\mathrm{id}],\gamma^{\gamma[1]}_{\gamma[1]-a}[\mathrm{id}]]$. Set $\mathcal{N}$ equal to the collection of precisely all the measures $\mu$ on $\mathcal{B}_{\JJ}$ with $\mu[1]=a$, $\mu[\mathrm{id}]=b$ and $\mu\leq \gamma$. Then if $\phi:\JJ\to\mathbb{R}$ is concave with $\gamma[\phi^-]<\infty$: 
\begin{enumerate}[(i)]
\item\label{proposition:general:i} the supremum $\sup_{\mu\in \mathcal{N}}\mu[\phi]$ is attained at the unique connected part of $\gamma$ that belongs to $\mathcal{N}$; 
\item\label{proposition:general:ii}  the infimum $\inf_{\mu\in \mathcal{N}}\mu[\varphi]$ is attained at the unique co-connected part of $\gamma$ that belongs to $\mathcal{N}$.
\end{enumerate}
If $\phi$ is strictly concave, then the supremum and infimum in the preceding are \emph{uniquely} attained.
\end{proposition}
\begin{remark}
By Jensen's inequality, $\gamma[\phi^+]<\infty$, which guarantees (together with the assumed condition $\gamma[\phi^-]<\infty$) that all the considered integrals are well-defined and finite.
\end{remark}
\begin{remark}\label{remark:ancillary}
The conditions on the pair $(a,b)$ simply guarantee that $\mathcal{N}$ is non-empty. In particular if $\nu$ is a finite measure on $\mathcal{B}_{\JJ^2}$ with second marginal ${X_2}_\star \nu$ equal to $\gamma$, then for any $x\in \JJ$, setting $a_x:=\nu(\{x\}\times \JJ)\in [0,\gamma[1]]$, one has $b_x:=\nu[X_2;X_1=x]\in [\gamma^{a_x}_{0}[\mathrm{id}],\gamma^{\gamma[1]}_{\gamma[1]-a_x}
[\mathrm{id}]]$, i.e. the pair $(a_x,b_x)$ satisfies the conditions for $(a,b)$ of Proposition~\ref{proposition:general}: indeed $\nu(X_2\in \cdot,X_1=x)\in\mathcal{N}$.
\end{remark}
\begin{proof}
If $b\in \{\gamma^{a}_{0}[\mathrm{id}],\gamma^{\gamma[1]}_{\gamma[1]-a}[\mathrm{id}]\}$ then $\mathcal{N}$ is a singleton consisting of a connected and co-connected part of $\gamma$, and there is nothing to prove. Then assume $b\in  (\gamma^{a}_{0}[\mathrm{id}],\gamma^{\gamma[1]}_{\gamma[1]-a}[\mathrm{id}])$. We prove case \ref{proposition:general:i}; case \ref{proposition:general:ii} follows by a simple modification of the argument.

\noindent
{\bf Existence of connected optimizer:}  Clearly 
\begin{equation*}
  \sup_{\mu\in \mathcal{N}}\mu[\phi]=\sup_{\mu\in \mathrm{M}}\inf_{\beta\in \mathbb{R},\alpha\in \mathbb{R}}\mu[\phi+\beta \mathrm{id}+\alpha]-\beta b-\alpha a,
\end{equation*}
where $\mathrm{M}$ is the collection of those measures $\mu$ on $\mathcal{B}_{\JJ}$ for which $\mu\leq \gamma$. This is bounded above by $\inf_{\beta\in\mathbb{R},\alpha\in \mathbb{R}}\sup_{\mu\in \mathrm{M}}\mu[\ln+\beta \mathrm{id}+\alpha]-\beta b-\alpha a$. Moreover if $\beta^*\in \mathbb{R},\alpha^*\in \mathbb{R},\mu^*\in \mathrm{M}$ can be found such that  $\mu^*$ is optimal for $\max_{\mu\in \mathrm{M}}\mu[\phi+\beta^*\mathrm{id}+\alpha^*]$, with $\mu^* [\mathrm{id}]=b$ and $\mu^*[1]=a$ (``Lagrange optimality conditions''), then there is ``minimax equality'' and $\mu^*$ attains the supremum in $\sup_{\mu\in \mathcal{N}}\mu[\phi]$. Taking for $\mu^*$ the unique connected part of $\mu$ that belongs to $\mathcal{N}$, cf. Remark~\ref{remark:connected-part}, we see that it suffices to choose $\beta^*$ and $\alpha^*$ in such a way that $\phi+\beta^*\mathrm{id}+\alpha^*$ is nonnegative on the smallest interval $\II$ that carries $\mu^*$ and nonpositive off this interval: of course automatically it must then vanish at the endpoints $x^*\leq X^*$ of said interval, which belong to $\JJ$ because we have assumed that $b\in (\gamma^{a}_{0}[\mathrm{id}],\gamma^{\gamma[1]}_{\gamma[1]-a}[\mathrm{id}])$. To this end set first  $\beta^*=-\frac{\phi(X^*)-\phi(x^*)}{X^*-x^*}$ when $x^*<X^*$, and take any $\beta^*\in [-\phi'_-(x^*),-\phi'_+(X^*)]$ when $x^*=X^*$;  then take $\alpha^*=-\phi(X^*)-\beta^* X^*$. By concavity of $\phi$ this gives the desired $\beta^*$ and $\alpha^*$.

\noindent
{\bf Uniqueness of optimizer:} Assume now $\phi$ is strictly concave. It is clear that any maximizer of $\mu[\phi]$ over $\mu\in \mathcal{N}$ also maximizes $\mu[\phi+\beta^*\id+\alpha^*]$ over $\mu\in \mathcal{N}$. But because of the strict concavity of $\phi$, $\phi$ is strictly positive on $(x^*,X^*)$ and strictly negative on $\JJ\backslash [x^*,X^*]$, which renders $\mu^*$ the unique maximizer of $\mu[\phi+\beta^*\id+\alpha^*]$ over $\mu\in \mathcal{N}$.
\end{proof}

We turn now to the notion of curtain martingale transports.

\begin{definition}\label{definition:curtain}
%Let $\lambda_1$ and $\lambda_2$ be probability measures on $\mathcal{B}_ {\JJ}$ of finite mean  (resp. and let $R$ be any relation on $\JJ$). Then: (i) a (resp. $R$-curtain) martingale coupling of $\lambda_1$ and $\lambda_2$ (or a (resp. $R$-curtain) martingale transport of $\lambda_1$ to $\lambda_2$) is a probability $\nu$ on $\mathcal{B}_{\JJ^2}$ such that ${X_1}_\star \nu=\lambda_1$, ${X_2}_\star\nu=\lambda_2$, $\nu[X_2\vert X_1]=X_1\text{ a.s.-}\nu$  (resp. and the following holds true: there exists a $\Gamma\in \mathcal{B}_{\JJ^2}$ that carries $\nu$ and such that one cannot have $\{(x',y'),(x,y_1),(x,y_2)\}\subset \Gamma$ with $xRx'\text{ and }y_1<y'<y_2$); (ii) a curtain martingale coupling (or curtain martingale transport) of $\lambda_1$ to $\lambda_2$ is a martingale coupling of $\lambda_1$ and $\lambda_2$ for which there exists  $\Gamma\in \mathcal{B}_{\JJ^2}$ that carries $\nu$ and such that one cannot have  $\{(x,y),(x,y_1),(x,y_2),(x',y'),(x',y_1'),(x',y_2')\}\subset \Gamma$ with $$x\ne x',\, y_1<y'<y_2\text{ and }y_1'<y<y_2'.$$ For brevity of expresssion sometimes we omit the qualification ``martingale'' in ``curtain martingale''.
Let $\lambda_1$ and $\lambda_2$ be probability measures on $\mathcal{B}_ {\JJ}$ of finite mean and let $R$ be any relation on $\JJ$. Then:
\begin{enumerate}
\item a \emph{martingale coupling} of $\lambda_1$ and $\lambda_2$ (or a martingale transport of $\lambda_1$ to $\lambda_2$) is a probability $\nu$ on $\mathcal{B}_{\JJ^2}$ such that ${X_1}_\star \nu=\lambda_1$, ${X_2}_\star\nu=\lambda_2$, $\nu[X_2\vert X_1]=X_1\text{ a.s.-}\nu$;
\item a \emph{$R$-curtain martingale coupling} of $\lambda_1$ and $\lambda_2$ (or a $R$-curtain martingale transport of $\lambda_1$ to $\lambda_2$) is a martingale coupling such that additionally there exists a $\Gamma\in \mathcal{B}_{\JJ^2}$ that carries $\nu$ and such that one cannot have $\{(x',y'),(x,y_1),(x,y_2)\}\subset \Gamma$ with $xRx'\text{ and }y_1<y'<y_2$;
\item a \emph{curtain martingale coupling} (or curtain martingale transport) of $\lambda_1$ to $\lambda_2$ is a martingale coupling of $\lambda_1$ and $\lambda_2$ for which there exists  $\Gamma\in \mathcal{B}_{\JJ^2}$ that carries $\nu$ and such that one cannot have  $\{(x,y),(x,y_1),(x,y_2),(x',y'),(x',y_1'),(x',y_2')\}\subset \Gamma$ with $$x\ne x',\, y_1<y'<y_2\text{ and }y_1'<y<y_2'.$$
\end{enumerate}
\end{definition}
\begin{remark}
%\cred{AC: To update? \cite{biegelbock}}
The terminology ``curtain martingale coupling'' comes from \citet[Theorem~4.5]{juillet} where the notions of a left- and of a right-curtain coupling were introduced. In fact, except that \cite{juillet} works on $\mathbb{R}$ where we allow a non-empty open interval of the real line, every left- (right-) curtain coupling in the sense of \cite{juillet} is a curtain martingale coupling in the sense of the preceding definition (but the converse is not true; cf. Remark~\ref{remark:curtains}).  Further, $<_\JJ$-curtain (resp. $>_\JJ$-curtain) couplings are precisely the left- (resp. right-) curtain couplings in the sense of \cite{juillet}. See also Remark~\ref{remark:curtains} below.
\end{remark}
Recall also:

\begin{definition}\label{def:convex}
Probabilities $\lambda_1$ and $\lambda_2$ on $\mathcal{B}_{\JJ}$ of finite mean are said to be in convex order provided $\lambda_1[\phi]\leq \lambda_2[\phi]$ for all convex $\phi:\JJ\to \mathbb{R}$. 
\end{definition}

\begin{remark}
In the preceding definition, by Jensen's inequality, automatically $\lambda_i[\phi^-]<\infty$, $i\in \{1,2\}$, so that all the integrals are well-defined; by Strassen's Theorem (\cite{strassen}), this is well-known to be equivalent to the existence of a martingale transport of $\lambda_1$ to $\lambda_2$. 
\end{remark}

We now state the following crucial constructive characterization of curtain martingale transports when the first marginal is finitely supported.

\begin{proposition}\label{proposition:curtain-transport}
Let $\lambda_1$ and $\lambda_2$ be probability measures on $\mathcal{B}_ {\JJ}$ of finite mean. Assume $\lambda_1$ has finite support; let $A:=\supp(\lambda_1)$ and $n:=\vert A\vert$. Then the following statements are equivalent for any given $\nu$:
\begin{enumerate}[(i)]
\item \label{curtain:i} $\nu$ is a curtain martingale coupling of $\lambda_1$ and $\lambda_2$. 
\item\label{curtain:ii}  There exists a bijection $J:\{1,\ldots,n\}\to A$ (i.e. an enumeration of $A$) such that $\nu=\delta_{J(1)}\times \nu_1+\cdots+\delta_{J(n)}\times \nu_{n}$, where inductively, for $i\in \{1,\ldots,n\}$, $\nu_i$ is the unique connected part of $\lambda_2-(\nu_1+\cdots+\nu_{i-1})$ of mass $\lambda_1(\{J(i)\})$ and first moment $\lambda_1(\{J(i)\})J(i)$ (in the terminology of \cite[Definition~2.4]{biegelbock}, $\nu_i$ is the shadow of $\lambda_1(\{J(i)\})\delta_{J(i)}$ in $\lambda_2-(\nu_1+\cdots+\nu_{i-1})$).
\end{enumerate}
Furthermore, if $\lambda_1$ and $\lambda_2$ are in convex order, then for each bijection $J:\{1,\ldots,n\}\to A$, there exists a (necessarily unique) $\nu$ satisfying \ref{curtain:ii} with this $J$.
\end{proposition}
\begin{remark}
With $J$ as in \ref{curtain:ii}, $\nu$ is a $<^J$-curtain transport, where $<^J$ is the strict total order relation on $A$ that satisfies $a<^J b$ iff $J^{-1}(a)<J^{-1}(b)$ for $\{a,b\}\subset A$. Conversely, if $R$ is a strict total order on $A$, then any $R$-curtain transport is a curtain transport of $\lambda_1$ to $\lambda_2$.%, and hence unique by the characterization of this proposition.
\end{remark}
\begin{remark}
Either condition implies that $\lambda_1$ and $\lambda_2$ are in convex order.
\end{remark}
\begin{remark}\label{remark:curtains}
The left- (right-) curtain coupling of \citet[Theorem~4.5]{juillet} corresponds to taking in \ref{curtain:ii} $J$ to be increasing (decreasing). On the other hand all curtain couplings are shadow couplings in the sense of \citet{biegelbock}. Specifically, if $J$ is as in \ref{curtain:ii}, then  $\nu$ is the shadow coupling of \cite[Theorem~1.1]{biegelbock} corresponding to the unique lift $\widehat{\lambda_1}$ \cite[p.~2]{biegelbock} of $\lambda_1$ that is concentrated on the graph of the function $T:(0,1)\to A$ specified as follows: for $i\in \{1,\ldots,n\}$, $T\vert_{(\sum_{k=1}^{i-1} \lambda_1(\{J(k)\}),\sum_{k=1}^{i} \lambda_1(\{J(k)\}))}=J(i)$,  the values of $T$ on (the Lebesgue measure null set) $A$ not being important. This can be seen most easily via the characterization of \cite[Theorem~1.1(3)]{biegelbock} (note that our substitution of $\JJ$ for what is $\mathbb{R}$ in \cite{biegelbock} is of no consequence).
\end{remark}
\begin{question}
When $\lambda_1$ is not necessarily finitely supported, are all curtain couplings still shadow couplings? If so, what is the corresponding lift measure? Given the preceding an ``obvious'' conjecture is that they correspond to lifts of the form $\widehat{\lambda_1}=(\mathrm{id}_{(0,1)},g)_\star \mathfrak{l}$, where $\mathfrak{l}$ is Lebesgue measure on $\mathcal{B}_{(0,1)}$, and $g$ ranges over $(\JJ,\mathcal{B}_\JJ)$-valued random variables on $((0,1),\mathcal{B}_{(0,1)},\mathfrak{l})$ with $g_\star \mathfrak{l}=\lambda_1$, whose preimages of singletons are (possibly empty, of course) intervals. To see this in one direction, let $g$ be such and let $\hat{\pi}$ and $\hat{\Gamma}$ be as in \cite[Theorem~1.1]{biegelbock}; we claim that $\pi:=(\mathrm{pr}_{23})_\star \hat{\pi}$, the associated shadow coupling, is a curtain transport, at least provided $g$ is even injective (which is automatic if $\lambda_1$ is diffuse). To see it note that by \cite[p. 143, (5g)]{halmos} we may assume that $\hat{\Gamma}\subset g\times \JJ$, since $g=\{(u,g(u)):u\in (0,1)\}\in \mathcal{B}_{(0,1)\times \JJ}$ carries $\widehat{\lambda_1}=(\mathrm{pr}_{12})_\star \hat{\pi}$. Then set $\Gamma':=\mathrm{pr}_{23}(\hat{\Gamma})$; it is not necessarily Borel, but it is analytic, and in particular it is universally measurable. Therefore there is a $\Gamma\subset \Gamma'$ such that $\Gamma\in \mathcal{B}_{\JJ^2}$ and $\overline{\pi}(\Gamma'\backslash \Gamma)=0$. In consequence $\pi( \Gamma)=\overline{\pi}(\Gamma')=\overline{\hat{\pi}}( (\mathrm{pr}_{23})^{-1}(\Gamma'))=\overline{\hat{\pi}}( (\mathrm{pr}_{23})^{-1}(\mathrm{pr}_{23}(\hat{\Gamma})))\geq \hat{\pi}(\hat{\Gamma})=1$, i.e. $\pi$ is supported by $\Gamma$. Now let $\{(x,y),(x,y_1),(x,y_2),(x',y'),(x',y_1'),(x',y_2')\}\subset \Gamma$, $x\ne x'$. Then there are $\{s,s_1,s_2,s',s'_1,s'_2\}\subset (0,1)$ with $\{(s,x,y),(s_1,x,y_1),(s_2,x,y_2),(s',x',y'),(s'_1,x',y_1'),(s'_2,x',y_2')\}\subset \hat{\Gamma}$. Because $\hat{\Gamma}\subset g\times \JJ$, and since $g$ is injective, we have $s=s_1=s_2$ and $s'=s'_1=s_2'$; automatically $s\ne s'$ since $x\ne x'$. Then \cite[Theorem~1.1(3)]{biegelbock} implies that $y'\notin (y_1,y_2)$ or $y\notin (y_1',y_2')$, according as $s<s'$ or $s'<s$. Thus by definition $\pi$ is a curtain transport.
\end{question}
\begin{example}\label{example:curtains}
It may happen that the curtain martingale transports corresponding to two (or indeed all) $J$ in \ref{curtain:ii} are the same (it is easy to see: take $\lambda_2=\lambda_1$), but in general it can also happen that the curtain martingale transports corresponding to distinct $J$ in \ref{curtain:ii} are all distinct. To see the latter, let $\JJ=\mathbb{R}$, $\lambda_1=\frac{1}{3}(\delta_{-1}+\delta_0+\delta_1)$ and $\lambda_2=\mathrm{Unif}([1,2]\cup [-2,-1])$. Then $\lambda_1$ and $\lambda_2$ are in convex order and clearly for distinct $J$ the curtain martingale transports described in \ref{curtain:ii} are distinct. 
\end{example}
\begin{proof}
If \ref{curtain:ii} holds, then clearly $\nu$ is a martingale coupling of $\lambda_1$ and $\lambda_2$; to check that it is a curtain martingale transport one may take $\Gamma=\cup_{i=1}^n\{J(i)\}\times \mathrm{supp}(\nu_i)$ in Definition~\ref{definition:curtain}. Thus \ref{curtain:ii} implies \ref{curtain:i}. The last statement of the proposition follows by an inductive application of \cite[Lemma~2.8]{juillet}.

Conversely, suppose $\nu$ is a curtain martingale coupling and let $\Gamma$ be as in Definition~\ref{definition:curtain}. For $a\in A$, set $O_a:=[\conv(X_2(\Gamma\cap [\{a\}\times \JJ]))]^\circ$: in words one looks at the range of $\Gamma$ out of $a$, generates the smallest interval that contains this set, and then takes its interior; of course $O_a$ may be empty. Now choose $b\in A$ such that $O_b$ is a maximal element of the set $\mathcal{O}:=\{O_a:a\in A\}$ with respect to reverse inclusion $\supset$. % (For instance, by Zorn's Lemma)
Let $A':=\{a\in A:O_{a}=O_b\}$. Then either there is an $a\in A'$ such that $O_a\cap X_2(\Gamma\cap [\{a\}\times \JJ])\ne \emptyset$ and fix such an $a$, else take for $a$ any element of $A'$. Because of the condition of Definition~\ref{definition:curtain} on $\Gamma$, if $a_1\ne a_2$ are elements of $A$, then either $O_{a_1}\cap O_{a_2}=\emptyset$ or else $O_{a_1}$ and $O_{a_2}$ are comparable with respect to inclusion. Using this and the condition on $\Gamma$ again, one sees that if $a'\in A\backslash \{a\}$ and $(a',y')\in \Gamma$, then $y'\notin O_a$. But $\Gamma$ carries $\nu$, which is a martingale transport of $\lambda_1$ to $\lambda_2$. Therefore $\nu(\{a\}\cap\cdot)$ is the unique connected part of $\lambda_2$ of mass  $\lambda_1(\{a\})$ and first moment $\lambda_1(\{a\})a$. Setting $J(1)=a$ and proceeding inductively (via obvious restrictions and renormalizations of $\nu$, $\lambda_1$, $\lambda_2$), we see that \ref{curtain:i} implies \ref{curtain:ii}, which concludes the proof.
\end{proof}

\begin{remark}
In Proposition~\ref{proposition:curtain-transport}, if one does not assume that the support of $\lambda_1$ is finite, then the proof of the  implication \ref{curtain:i}$\Rightarrow$\ref{curtain:ii} breaks down when one chooses a maximal element of $\mathcal{O}$ with respect to $\supset$: every finite linearly ordered subset of $\mathcal{O}$ admits an upper bound in $\mathcal{O}$,  but this is of course no longer true for even just countably infinite sets (so that one could apply Zorn's lemma). In fact we have the following counter-example: Let $(p_i)_{i\in \mathbb{N}}$ be a sequence in $(0,1)$ with $\sum_{i\in \mathbb{N}}p_i=1$. Furthermore, let for each $i\in \mathbb{N}$, $\nu_i$ be a probability on $\mathcal{B}_{\JJ}$, whose support is $[1-1/n,1+1/n]\backslash (1-1/(n+1),1+1/(n+1))$, and whose mean is $1+1/(2n)$. Clearly such constellations (and any number of others that would be just as good) obtain. Then $\sum_{i\in \mathbb{N}}p_i\delta_{1+1/(2n)}\times \nu_i$ is a curtain martingale transport of $\sum_{i\in \mathbb{N}} p_i\delta_{1+1/(2n)}$ to $\sum_{i\in \mathbb{N}}p_i\nu_i$, yet none of the $p_i\nu_i$, $i\in \mathbb{N}$, is a connected part of $\sum_{i\in \mathbb{N}} p_i\nu_i$. (This being so, when the support of $\lambda_1$ is countably infinite, then \cite[Lemma~2.8]{juillet} can still be applied inductively to see that the procedure of \ref{curtain:ii}, with the obvious modifications for the denumerably infinite case, produces a curtain martingale transport.)
\end{remark}
%\cred{Include Figure?}
Before giving the last main result of this section, that will connect curtain martingale transports to martingale transport optimization, we must prepare some further groundwork. 

The following auxiliary result is a very special case of \cite[Lemma~1.11]{juillet}. In it, and in the sequel, for real-valued functions $f$ and $g$ defined on $\JJ$, we use the notation $g\otimes f$ to mean the tensor product function defined on $\JJ^2$ with the values $(g\otimes f)((x,y))=g(x)f(y)$ for $(x,y)\in \JJ^2$.

\begin{lemma}\label{lemma:variational}
Let $\lambda_1$ and $\lambda_2$ be probabilities on $\mathcal{B}_{\JJ}$, $g\in b\mathcal{B}_{\JJ}$ and $f\in  L^1(\lambda_2)$. Assume $\pi$ is a martingale transport of $\lambda_1$ to $\lambda_2$ that maximizes $\nu[g\otimes f]$ over all martingale transports $\nu$ of $\lambda_1$ to $\lambda_2$. Then there exists a $\Gamma\in \mathcal{B}_{\JJ^2}$ that carries $\pi$ and such that the following holds: whenever $\alpha$ is a probability on $\mathcal{B}_{\JJ^2}$ with finite support contained in $\Gamma$, then $\alpha'[g\otimes f]\leq\alpha[g\otimes f]$ for any probability  $\alpha'$ on $\mathcal{B}_{\JJ^2}$ for which  ${X_1}_\star  \alpha={X_1}_\star  \alpha'$, ${X_2}_\star \alpha={X_2}_\star  \alpha'$ and $\alpha_x[X_2]=\alpha_{x}'[X_2]$ for ${X_1}_\star\alpha$-a.e. $x\in \JJ$, where $(\alpha_x)_{x\in \JJ}$ and $(\alpha_{x}')_{x\in \JJ}$ are disintegrations of $\alpha$ and $\alpha'$ with respect to ${X_1}_\star\alpha$, respectively.
\end{lemma}
\begin{proof}
The setting of  \citet{juillet} is the real line, where we allow $\JJ$ to be an open interval of $\mathbb{R}$. However, if necessary, i.e. when $\JJ\ne \mathbb{R}$, it is straightforward to extend all the measures and functions to $\mathbb{R}$ by setting them equal to $0$ off $\JJ$, and $\pi$ remains optimal in this extended setting. Modulo this the sufficient integrability condition of \cite{juillet} is met because $g$ is bounded and $f$ is integrable for $\lambda_2$, while $\pi$  leads to a finite value of $\pi[g\otimes f]$ for the very same reasons. As a consequence \cite[Lemma~1.11]{juillet} applies, and is easily translated back to the interval $\JJ$ in place of $\mathbb{R}$, if  necessary (i.e. when $\JJ \ne\mathbb{R}$).
\end{proof}
We have next a result which maintains that any optimizer  of the  ``linear'' optimal martingale transport problem, in the special case when the objective functional is of tensor product form, is always a curtain martingale transport; in precise terms: 

\begin{proposition}\label{proposition:curtains-optimal}
Let $\lambda_1$ and $\lambda_2$ be probabilities on $\mathcal{B}_{\JJ}$ of finite mean. 
\begin{enumerate}[(i)]
\item \label{proposition:curtains-optimal:i} Let also $g\in b\mathcal{B}_\JJ$ and $f:\JJ\to \mathbb{R}$ be strictly convex with $\lambda_2[f^+]<\infty$. Assume furthermore $\pi$ is a martingale transport of $\lambda_1$ to $\lambda_2$ that maximizes $\nu[g\otimes f]$ over all martingale transports $\nu$ of $\lambda_1$ to $\lambda_2$. Then $\pi$ is a $<_g$-curtain martingale transport of $\lambda_1$ to $\lambda_2$, where $<_g$ is the  relation on $\JJ$ for which $a<_gb$ iff $g(a)<g(b)$; in particular, if $g$ is injective on a set that carries $\lambda_1$, then $\pi$ is a curtain transport.
\item\label{proposition:curtains-optimal:ii} Suppose now $\lambda_1$ has a finite support and let $\pi$ be a curtain martingale transport of $\lambda_1$ to $\lambda_2$. Then, conversely, for any $f:\JJ\to \mathbb{R}$ that is strictly concave with $\lambda_2[f^-]<\infty$, there exists $g\in b\mathcal{B}_\JJ$ injective on, and vanishing off $\supp(\lambda_1)$, such that $\pi$ uniquely maximizes $\nu[g\otimes f]$ over all martingale transports $\nu$ of $\lambda_1$ to $\lambda_2$. 
\end{enumerate}
\end{proposition}
\begin{remark}
Note that in \ref{proposition:curtains-optimal:i} we may take, ceteris paribus, $f:\JJ\to \mathbb{R}$ strictly concave with $\lambda_2[f^-]<\infty$ and/or replace ``maximizes'' with ``minimizes'', and still the conclusion that $\pi$ is a curtain transport provided $g$ is injective on a set that carries $\lambda_1$ remains valid (replace $g$ with $-g$ and/or $f$ with $-f$). A similar observation pertains to \ref{proposition:curtains-optimal:ii}.
\end{remark}

\begin{proof}
\ref{proposition:curtains-optimal:i}. Suppose, for a contradiction,  that $\pi$ is not a $<_g$-curtain martingale transport. Let $\Gamma$ be as in Lemma~\ref{lemma:variational} and take  $D:=\{(x',y'),(x,y_1),(x,y_2)\}\subset \Gamma$ with $g(x)<g(x')\text{ and }y_1<y'<y_2$.

 Further let $\alpha$ be any probability on $\mathcal{B}_{\JJ^2}$ with support $D$ --- such probabilities certainly exist --- and consider the problem $\max \alpha'[g\otimes f]$ over all probabilities $\alpha'$ that have the same marginals as $\alpha$ and the same conditional first moments of $X_2$  given $X_1$ as $\alpha$ (as in Lemma~\ref{lemma:variational}). % We want to see that we can do strictly better than $\alpha[g\otimes f]$, because this will contradict Lemma~\ref{lemma:variational}. 
Denote $p_x:={X_1}_\star \alpha(\{x\})=\alpha(X_1=x)$ and  $p_{x'}:={X_1}_\star \alpha(\{x'\})=\alpha(X_1=x')$. Then we are maximizing $p_xg(x)\beta_x[f]+p_{x'}g(x')\beta_{x'}[f]$ over probabilities $\beta_x$ and $\beta_ {x'}$ on $\mathcal{B}_{\JJ}$ with $\beta_x[\mathrm{id}]=\alpha[X_2;X_1=x]/p_x$, $\beta_{x'}[\mathrm{id}]=\alpha[X_2;X_1=x']/p_{x'}$ and $p_x\beta_x+p_{x'}\beta_{x'}={X_2}_\star \alpha=:\alpha_2$. In other words we are maximizing $(g(x)-g(x'))\mu[f]$ over measures $\mu$ on $\mathcal{B}_{\JJ}$ with $\mu[1]=\alpha(X_1=x)$, $\mu[\mathrm{id}]=\alpha[X_2;X_1=x]$ and $\mu\leq\alpha_2$. Furthermore, because $g(x)<g(x')$, it is the same as minimizing $\mu[f]$ over the specified class of $\mu$. 

Then, on the one hand, by Lemma~\ref{lemma:variational}, $\alpha(X_2\in \cdot;X_1=x)$ is an optimizer in the preceding. On the other hand, by Proposition~\ref{proposition:general} coupled with Remark~\ref{remark:ancillary}, the unique optimizer to this problem identified there violates the property of $\alpha$ having support $D$, that is, $\alpha(X_2\in\cdot;X_1=x)$ is not a connected part of $\alpha_2$. 

The final statement of this part is a consequence of the fact that  if $g$  is injective on a set that carries $\lambda_1$, then every $<_g$-curtain transport is a curtain transport.

\ref{proposition:curtains-optimal:ii}. Let the support of $\lambda_1$ be the set $A=\{a_1,\ldots,a_n\}$ of size $n\in\mathbb{N}$, let $\pi$ correspond to an enumeration  $J$ of $A$, as in Proposition~\ref{proposition:curtain-transport}\ref{curtain:ii}, and let $\mathcal{C}$ be the collection of all curtain transports of $\lambda_1$ to $\lambda_2$, which is finite by Proposition~\ref{proposition:curtain-transport}. For $c\in\mathcal{C}$ write $(c_x)_{x\in \JJ}$ for the disintegration of $c$ against $\lambda_1$. Set $g$ to be equal to zero off $\supp(\lambda_1)$, without limiting oneself insist further that $g$ is nonnegative with $\sum g=1$, and define the values $g(a_i)$, $i\in \{1,\ldots,n\}$, as follows: let $M:=\pi_{a_{J(1)}}[f]$ and \[m:=\max\{c_{a_{J(1)}}[f]:c\in \mathcal{C}\text{ such that }c_{a_{J(1)}}\text{ is not a connected part of }\lambda_2\}.\] By Proposition~\ref{proposition:general}, $m<M$. Take $g(a_{J(1)})$ close to, but strictly less than $1$, and in any event more than $1/2$, so that no matter what the values of $g$ on $A\backslash \{a_{J(1)}\}$ (save for the requirement of injectivity), any maximizer of $\nu[g\otimes f]$ over all martingale transports $\nu$ of $\lambda_1$ to $\lambda_2$ is attained at a curtain transport $c$ for which $c_{a_{J(1)}}$ is a connected part of $\lambda_2$.  Note indeed that any such maximizer is necessarily a curtain transport by \ref{proposition:curtains-optimal:i}.

We may now iterate this inductively in the obvious manner (discarding $\lambda_1(a_{J(1)})\delta_{a_{J(1)}}$ from $\lambda_1$, $\lambda_1(a_{J(1)})\delta_{a_{J(1)}}\times \pi_{a_{J(1)}}$ from $\lambda_2$, and renormalising), in order to arrive at a $g$ such that any maximizer of $\nu[g\otimes f]$ over all martingale transports $\nu$ of $\lambda_1$ to $\lambda_2$ is equal to $\pi$; there is such a maximizer because the set of all martingale  transports of $\lambda_1$ to $\lambda_2$ is weakly closed, while the map $\nu\mapsto \nu[g\otimes f]$ is weakly continuous; this is easy to see because $\lambda_1$ is finitely supported. The map $g$ obtained in this manner satisfies all the requisite properties (the injectivity of $g$ on $\supp(\lambda_1)$ is ensured by choosing $g(a_{J(1)})>1/2$ and then keeping to this convention in each inductive step).
\end{proof}
Combining Propositions~\ref{proposition:curtains-optimal} and~\ref{proposition:curtain-transport} we have

\begin{corollary}\label{corollary:curtain}
In the setting of Proposition~\ref{proposition:curtain-transport} the statements \ref{curtain:i}  and \ref{curtain:ii} are further equivalent to
\begin{enumerate}[(i)]
\setcounter{enumi}{2}
\item\label{curtain:iii} For some $f:\JJ\to \mathbb{R}$ that is strictly concave with $\lambda_2[f^-]<\infty$ and some $g\in b\mathcal{B}_\JJ$ injective on $A$,  $\nu$ (uniquely) maximizes $\rho[g\otimes f]$ over all martingale transports $\rho$ of $\lambda_1$ to $\lambda_2$. \qed
\end{enumerate}
\end{corollary}

\begin{remark}
This dovetails with \cite[Theorem~1.1(1), 3rd bullet point on p.~3, Remark~4.6]{biegelbock} (cf. Remark~\ref{remark:curtains}).
\end{remark}

In addition, we can prove the following technical lemma:
\begin{lemma}\label{lemma:concave}
For any given concave $f:\JJ\to \mathbb{R}$, there exists a sequence $(f^m)_{m\in \mathbb{N}}$ of strictly concave functions, mapping $\JJ\to \mathbb{R}$, and such that $\vert f^m(x)-f(x)\vert \leq (a\vert x\vert +b)/m$ for all $x\in \JJ$ and $m\in \mathbb{N}$, for some $\{a,b\}\subset[0,\infty)$.
\end{lemma}
\begin{proof}
Fix an $x_0\in \JJ$ and take any function $D\in \mathcal{B}_\JJ/\mathcal{B}_{\mathbb{R}}$, strictly decreasing and bounded (such functions certainly exist). Then set $f^m(x):=f(x_0)+\int_{x_0}^x(f'_-(y)+\frac{D(y)}{m})dy=f(x)+\frac{1}{m}\int_{x_0}^xD(y)dy$ for $x\in \JJ$, $m\in \mathbb{N}$.  
\end{proof}
\begin{remark}
It is clear that in the preceding proof, if $\JJ\ne \mathbb{R}$ (but not otherwise), then $D$ can be chosen even interable w.r.t. Lebesgue measure, and one gets  uniform convergence of $(f_m)_{m\in \mathbb{N}}$ to $f$. %For our purposes the ``$\id$-uniform'' convergence that we get in general will suffice. 
\end{remark}
We now state the main result of this section.
\begin{theorem}\label{corollary:extremal-points}
Let $\lambda_1$ and $\lambda_2$ be in convex order and let $f:\JJ\to \mathbb{R}$ be concave with $\lambda_2[f^-]<\infty$. Assume $\lambda_1$ is supported by the finite set $\{a_1,\ldots,a_n\}$ of cardinality $n\in \mathbb{N}$. Set  $\mathcal{M}:= \{\text{martingale transports of }\lambda_1\text{ to }\lambda_2\}$ and let for $\nu\in \mathcal{M}$, $\nu_i:=\nu(X_2\in \cdot \, \vert X_1=a_i)$, $i\in \{1,\ldots,n\}$, be the disintegration of $\nu$ against $\lambda_1$. Then $\mathcal{X}:=\{(\nu_1[f],\ldots,\nu_n[f]):\nu\in \mathcal{M}\}$ is a convex compact subset of $\mathbb{R}^n$, which is equal to the convex hull of $\mathcal{E}:=\{(\nu_1[f],\ldots,\nu_n[f]):\nu\text{ a curtain martingale transport of $\lambda_1$ to $\lambda_2$}\}$.
\end{theorem}
\begin{remark}
We may take, ceteris paribus, $f:\JJ\to \mathbb{R}$ convex with $\lambda_2[f^+]<\infty$, and the same conclusion remains valid.
\end{remark}
\begin{example}\label{example:no-improvement}
The conclusion of Theorem~\ref{corollary:extremal-points} cannot be improved in the sense that in general all the points of $\mathcal{E}$ will be vertices of $\mathcal{X}$. To see this take $\lambda_1$ and $\lambda_2$ as in Example~\ref{example:curtains} and $f=\mathrm{id}_\mathbb{R}^2$. We may take the enumeration $a_1=-1$, $a_2=0$ and $a_3=1$. Note that for $b\in \mathbb{R}^3$ and $\nu\in \mathcal{M}$, $b\cdot (\nu_1[f],\nu_2[f],\nu_3[f])=\sum_{i=1}^3b_i\nu_i[\mathrm{id}_\mathbb{R}^2]$. As $b$ ranges over all the permutations of $\{-1,0,1\}$, it then follows from Proposition~\ref{proposition:general},  that the linear functional $x\mapsto b\cdot x$ will attain its unique maximum on  $\mathcal{X}$ at all of the members of $\mathcal{E}$, i.e. for each permutation $b$ of $\{-1,0,1\}$, a different member of $\mathcal{E}$ will be the unique maximum of $x\mapsto b\cdot x$ on $\mathcal{X}$. Therefore in this case each member of $\mathcal{E}$ is a vertex of $\mathcal{X}$.
\end{example}

\begin{remark}
More generally, if $f$ is strictly concave, it follows from Proposition~\ref{proposition:curtains-optimal}\ref{proposition:curtains-optimal:ii}, that all the members of $\mathcal{E}$ are vertices of $\mathcal{X}$.%, and further from the constructive characterization of curtain transports of Proposition~\ref{proposition:curtain-transport} and from Proposition~\ref{proposition:general} that the map $\nu\mapsto (\nu_1[f],\ldots,\nu_n[f])$ is injective over the curtain transports of $\lambda_1$ to $\lambda_2$. \textbf{Check!!}
\end{remark}
% \begin{question}
% When $f$ is strictly concave, is the map $\nu\mapsto (\nu_1[f],\ldots,\nu_n[f])$ injective  over the curtain transports of $\lambda_1$ to $\lambda_2$?
% \end{question}
\begin{proof}
It is clear that $\mathcal{X}$ is a convex subset of $\mathbb{R}^n$. To see that it is compact we argue as follows. 
\begin{itemize}
\item The map $\mathcal{M}\ni\nu \mapsto (\nu_1[f],\ldots,\nu_n[f])\in \mathbb{R}^n$  is continuous in the weak topology on $\mathcal{M}$: first one sees that $\mathcal{M}\ni \nu\mapsto \nu_i$ is weakly continuous for each $i\in \{1,\ldots,n\}$; then to account for  $f$ being just integrable for $\lambda_2$ and not bounded (as a concave map on an open interval, certainly it is continuous), approximate $f$ by truncating it, and exploit the fact that the $\nu_i$, $i\in \{1,\ldots,n\}$, are bounded by $\lambda_2/(\lambda_1(\{a_1\})\land \cdots\land \lambda_1(\{a_n\}))$, uniformly in $\nu\in \mathcal{M}$.
\item By a similar token, $\mathcal{M}$ is weakly closed in the set of all probability measures on $\mathcal{B}_{\JJ^2}$; besides, it is weakly relatively compact therein by Prokhorov's Theorem, since it is tight, which latter fact comes finally for instance from the simple estimate $\nu\leq c\times \lambda_2 $ for $\nu\in \mathcal{M}$, where $c$ is the counting measure on $\{a_1,\ldots,a_n\}$. 
\end{itemize}
By the preceding $\mathcal{X}$ is the continuous image of a compact set, thus compact.

It remains to argue that $\mathcal{X}=\conv(\mathcal{E})$. In order to verify this, assume $f$ is strictly concave in the first instance. We show first that every linear functional on $\mathbb{R}^n$ reaches its maximum on $\mathcal{X}$ in a point of $\mathcal{E}$. Then let $b\in \mathbb{R}^n$. We want to show that $\sup_{x\in \mathcal{X}}\langle b,x\rangle$  is attained in a point of $\mathcal{E}$. By a continuity argument we may assume that the $\frac{b_i}{\lambda_1(\{a_i\})}$, $i\in \{1,\ldots,n\}$, are  pairwise distinct. Then, for $\nu\in \mathcal{M}$, we can write $\langle b,(\nu_1[f],\ldots,\nu_n[f])\rangle=\sum_{i=1}^nb_i\nu_i[f]=\nu[g\otimes f]$, where $g(a_i):=\frac{b_i}{\lambda_1(\{a_i\})}$ for $i\in \{1,\ldots,n\}$ and $g$ vanishes off $\supp(\lambda_1)$; Proposition~\ref{proposition:curtains-optimal}\ref{proposition:curtains-optimal:i} yields the conclusion. Suppose now per absurdum that the convex hull of $\mathcal{E}$ is strictly smaller than $\mathcal{X}$; let $x\in \mathcal{X}\backslash \conv(\mathcal{E})$. By the hyperplane separation theorem, there is a $b\in \mathbb{R}^n$ such that $\langle b,x\rangle$ is strictly larger than $\max_{y\in \conv(\mathcal{E})}\langle b,y\rangle$. But that means that there is a linear functional on $\mathbb{R}^n$ that is not maximized on $\mathcal{X}$ by a point in $\mathcal{E}$, a contradiction.

Now suppose $f$ is merely concave and let $\CC$ be the set of curtain martingale transports of $\lambda_1$ to $\lambda_2$. In order to show that still $\mathcal{X}=\conv(\mathcal{E})$, let, via Lemma~\ref{lemma:concave},  $(f^m)_{m\in \mathbb{N}}$ be a sequence of strictly concave functions, mapping $\JJ\to \mathbb{R}$, uniformly integrable w.r.t. $\lambda_2$ and converging to $f$. It is only non-trivial to argue that each $x\in \mathcal{X}$ is a convex combination of the elements of $\mathcal{E}$. But, by what we have just  shown above, given a $\nu\in \mathcal{M}$, there exists a probability mass function (p.m.f.) $\lambda^m=(\lambda^m_c)_{c\in \CC}$ on $\CC$, such that $$(\nu_1[f^m],\ldots,\nu_n[f^m])=\sum_{c\in \mathcal{C}}\lambda^m_c(c_1[f^m],\ldots,c_n[f^m]).$$ The set of p.m.f. on $\CC$ being compact, by passing to a subsequence if necessary, we may assume that $\lambda^m$ converges (pointwise) to a p.m.f. $\lambda$ on $\CC$ as $m\to\infty$. Then one can pass to the limit $m\to\infty$ in the preceding display by dominated convergence (because of the uniform integrability of $(f^m)_{m\in \mathbb{N}}$ w.r.t. $\lambda_2$ /and hence w.r.t. all the $\nu_i$, $c_i$ involved in this display/).
\end{proof}

\begin{question}
If in Theorem~\ref{corollary:extremal-points} $\lambda_1$ is not necessarily supported by a finite set, and for a $\nu \in \mathcal{M}$ one defines the element $\nu^*:=(\JJ\ni x\mapsto \nu_x[f])$ of $L^1(\lambda_1)$, where $(\nu_x)_{x\in \JJ}$ is the disintegration of $\nu$ against $\lambda_1$, then one might well ask whether or not/conjecture that the convex set $\mathcal{X}:=\{\nu^*:\nu\in \mathcal{M}\}$ is the convex hull, in $L^1(\lambda_1)$, of $\{\nu^*:\nu\text{ a curtain martingale transport of $\lambda_1$ to $\lambda_2$}\}$. It is not immediately clear, however, how the above proof could be extended to cover this more general situation. Moreover, the practical usefulness of such a result would presumably be quite limited: (i) we have no procedure by means of which to determine all (or indeed any of) the curtain martingale transports when $\lambda_1$ is not finitely (countably) supported; (ii) even granted those, in the context of maximizing $\lambda_1[G(\nu^*)]$ over $\nu\in \mathcal{M}$ for a suitable map $G:\mathcal{X}\to \mathbb{R}$ (cf. Meta-corollary~\ref{remark:fundamental}), one is still looking at the difficult problem of optimizing $\lambda_1[G(f)]$ over $f$ belonging to the convex hull of the (what will presumably typically be infinite) set of curtain martingale transports of $\lambda_1$ to $\lambda_2$ (though at least the latter should be orders of magnitude easier than a direct optimization over $\mathcal{M}$).
\end{question}

\begin{meta-corollary}\label{remark:fundamental}
Retain the conditions of Theorem~\ref{corollary:extremal-points} and assume there is given a map $G:\mathcal{X}\to \mathbb{R}$. Then the optimization of $J(\nu):=G(\nu_1[f],\ldots,\nu_n[f])$ over $\nu\in \mathcal{M}$, at least to the extent of finding \emph{an} optimizer, reduces to the two subproblems:
\begin{enumerate}[(1)]
\item\label{reduce:one} the identification of the (finite) set $\mathcal{C}$ of curtain martingale transports (and correspondingly of the set $\mathcal{E}$), through Proposition~\ref{proposition:curtain-transport}; followed by 
\item\label{reduce:two} the ``classical'' optimization of the function $G$ over the compact convex set $\mathcal{X}=\conv(\mathcal{E})$.
\end{enumerate}
``Reduces'' in the sense that if $x$ maximizes (minimizes) $G$ over $\mathcal{X}$, then $x$ may be written as a convex combination $x=\sum_{\nu\in \mathcal{C}}\rho_\nu(\nu_1[f],\ldots,\nu_n[f])$, whence $\sum_{\nu\in \mathcal{C}}\rho_\nu \nu$ maximizes (minimizes) $G$ over $\mathcal{M}$. 
\qed
\end{meta-corollary}
%\begin{remark}
%The assumption that $G$ is continuous is only to ensure that the optimization in \ref{reduce:two} has  a solution. 
%\end{remark}
\begin{remark}\label{remark:meta-corollary}
By means of Meta-corollary~\ref{remark:fundamental}, the optimization of a great variety of martingale transport problems when the first marginal is finitely supported is, at least in a sense, canonically reduced to two separate problems. The first, \ref{reduce:one}, is ``universal'',  independent of the specifics of the optimization problem, and enabled by the identification of Proposition~\ref{proposition:curtain-transport}. The second, \ref{reduce:two}, is specific to the given problem, but it consists simply in the optimization of a function over a compact convex polytope (whose vertices belong to a known finite set) of an Euclidean space. When the problem is one of minimization and $G$ is concave, then an optimizer can be found amongst the vertices of the polytope. The main drawback when it comes to the practical implementation of this programme is that in Proposition~\ref{proposition:curtain-transport} the size of the space of martingale transports that one must check against is $n!$, where $n$ is the size of the support of the first marginal, and this grows prohibitively fast as $n$ increases. %We discuss this further in Section~\ref{section:numerical}. 
\end{remark}
%\cred{AC: Say something about `simplicity' of second problem? Gradient Descent?}
%\begin{proof}
%If $x$ maximizes (minimizes) $G$ over $\mathcal{X}$, then $x$ may be written as a convex combination $x=\sum_{\nu\in \mathcal{C}}\rho_\nu(\nu_1[f],\ldots,\nu_n[f])$, whence $\sum_{\nu\in \mathcal{C}}\rho_\nu \nu$ maximizes (minimizes) $J$ over $\mathcal{M}$. 
%\end{proof}

\section{A family of non-linear martingale transport optimizations}\label{section:main}
In this section we turn our attention to the family of problems \eqref{mtg-problem-intro}. Throughout, let $\gamma:\JJ\to \mathbb{R}$ be convex and $\phi:[0,\infty)\to \mathbb{R}$ be %strictly
 concave. Still $\JJ$ is a non-empty open interval of $\mathbb{R}$.
\subsection{Introducing the family of problems and some general considerations}\label{subsection:intro}
%, continuous
% and injective. 
%When considering the VIX, we will specialize to $\JJ=(0,\infty)$, $\gamma=-\frac{2}{\tau}\ln$ and $\phi=\sqrt{\cdot}$, but for now we keep it more general, since it will result in no greater difficulty to the arguments. 
As already indicated in the Introduction, recalling it  here for the reader's convenience, we will consider, for $\mu_1$ and $\mu_2$, probability measures on $\mathcal{B}_{\JJ}$ of finite mean, for which $\mu_1[\gamma^+]<\infty$, $\mu_2[\gamma^+]<\infty$, and in convex order, the optimization problem

\begin{equation}\label{eq:opt}
\max_{\nu\in \mathcal{M}}J(\nu),\text{ where }J(\nu):=\nu[V_\nu]
\text{ with } V_\nu:=\phi(\nu[\gamma(X_2)\vert X_1]-\gamma(X_1))
\end{equation}
$$\text{ for }\nu \in \mathcal{M}:=\{\text{martingale transports of $\mu_1$ to $\mu_2$}\}.$$

\begin{remark}
By Jensen's inequality this is all well-defined. Because $\mu_1$ and $\mu_2$ are in convex order,  $\MM$ is non-empty.
\end{remark}

\begin{remark}
Of course one can also look at the analogue of \eqref{eq:opt} with $\min$ replacing $\max$. We will make suitable remarks, where it will not be anyway obvious, to what extent the analysis carries over to cover the situation of minimization: in fact it will be so only when $\mu_1$ has a finite support, whereas the remainder of our arguments depend quite delicately on the problem being one of maximization. 
\end{remark}

\begin{remark}
We take the point of view that the quantity $J(\nu)$ stands for $\nu[\phi(\nu[\gamma(X_2)\vert X_1]-\gamma(X_1))]$ whenever this expression is well-defined (even if $\nu\notin \mathcal{M}$).
\end{remark}

\begin{definition}
Let $\kappa_1$, $\kappa_2$ and $\kappa_3$ be probabilities on $\mathcal{B}_\JJ$. Given a martingale transport $\rho$ from $\kappa_1$ to $\kappa_2$ and a martingale transport $\eta$ from $\kappa_2$ to $\kappa_3$, let $(\rho_x)_{x\in \JJ}$ (resp. $(\eta_x)_{x\in \JJ}$) be a disintegration of $\rho$ (resp. $\eta$) against $\kappa_1$ (resp. $\kappa_2$). Then we define $\rho\star\eta$ to be the martingale transport of $\kappa_1$ to $\kappa_3$ with disintegration against $\kappa_1$ given by the family $(\int\eta_y(\cdot)\rho_x(dy))_{x\in \JJ}$.
\end{definition}
The next proposition gathers some basic properties of the family of problems \eqref{eq:opt}. In particular item \ref{trivia:ii} identifies a monotonicity property of $ \sup_{\nu\in \mathcal{M}}J(\nu)$ in the first marginal $\mu_1$ (relative to the convex order of measures) that will later be instrumental in the proof of an ``approximation'' theorem (Theorem~\ref{theorem:general} below).
\begin{proposition}\label{proposition:trivia}
We have the following assertions:
\begin{enumerate}[(i)]
\item\label{trivia:i} $\MM$ is a weakly compact convex set. 
\item\label{trivia:ii} Assume $\phi$ is nondecreasing on $[0,2\Vert\gamma\Vert_\infty)$. Suppose $\mu_1'$ is another probability measure on $ \mathcal{B}_{\JJ}$ in convex order w.r.t. $\mu_1$ (and hence $\mu_2$) and let $\MM'$  be defined as $\MM$ above but with $\mu_1'$ replacing $\mu_1$. Then, if $\rho$ is a martingale transport of $\mu_1'$ to $\mu_1$, one has that $J(\rho\star \nu)\geq J(\nu)$ for all $\nu\in \MM$. In particular $\sup_{\nu'\in \mathcal{M}'}J(\nu')\geq \sup_{\nu\in \mathcal{M}}J(\nu)$.
\item\label{trivia:ii'}  Again assume $\phi$ is nondecreasing on $[0,2\Vert\gamma\Vert_\infty)$. Suppose $\mu_2'$ is another probability measure on $ \mathcal{B}_{\JJ}$, with $\mu_2$ in convex order w.r.t. $\mu_2'$, and let $\MM'$ be defined as $\MM$ above but with $\mu_2'$ replacing $\mu_2$. Then, if $\rho$ is a martingale transport of $\mu_2$ to $\mu_2'$, one has that $J(\nu\star \rho)\geq J(\nu)$ for all $\nu\in \MM$. In particular $\sup_{\nu'\in \mathcal{M}'}J(\nu')\geq \sup_{\nu\in \mathcal{M}}J(\nu)$.
\item\label{trivia:iii} Assume $\phi$ is continuous. If $\mu_1$ is carried by a finite set $S$, or if it is carried by a denumerable set $S$ having no limit points in $\JJ$ and $\phi$ is bounded, then the functional $J$ is continuous in the weak topology on $\mathcal{M}$. 
\end{enumerate}
\end{proposition}
\begin{proof}
\ref{trivia:i}. It is clear that $\MM$ is convex. Next, $\MM$ is weakly closed in the set of all probability measures on $\mathcal{B}_{\JJ^2}$. Indeed let $(\nu_n)_{n\in \mathbb{N}}$ be a sequence in $\MM$ with $\nu_n$ converging weakly to some probability $\nu_0$ on $\mathcal{B}_{\JJ^2}$ as $n\to\infty$. Then for any bounded, continuous $g:\JJ\to \mathbb{R}$, we have $\mu_1[g]=\nu_n[g\otimes 1]\to \nu_0[g\otimes 1]$ as $n\to\infty$, and likewise for the second marginal; in addition, in the equality $$\nu_n[X_2g(X_1)]=\nu_n[X_1g(X_1)]$$ one can pass to the limit $n\to \infty$ by a truncation of $X_2$ on the left-hand side and of $X_1$ on the right-hand side, exploiting the fact that the $\nu_n$, $n\in \mathbb{N}_0$, have fixed marginals that admit finite first moments. To see the latter for the left-hand side of  the preceding display, note that, with $X_2^m:=(X_2\land m)\lor (-m)$, one has \[\vert \nu_n[X_2g(X_1)]-\nu_0[X_2g(X_1)]\vert\leq 2\Vert g\Vert_\infty \mu_2[\vert \mathrm{id} \vert;\JJ\backslash [-m,m]]+\vert \nu_n[X_2^mg(X_1)]-\nu_0[X_2^mg(X_1)]\vert\] for $n\in \mathbb{N}$ and $m\in [0,\infty)$; then let $n\to\infty$ and $m\to \infty$ (in this order). Similarly for the right-hand side. The set $\MM$ is also weakly relatively compact in the set of all probabilities on $\mathcal{B}_{\JJ^2}$. This is by Prohorov's theorem, where tightness comes from the fact that the members of $\MM$ have fixed marginals: given any $\epsilon\in (0,\infty)$, there are compact $A$ and $B$ in $\JJ$ with $\mu_1[\JJ\backslash A]\leq \epsilon/2$ and $\mu_2[\JJ\backslash B]\leq \epsilon/2$; then $\nu[\JJ^2\backslash (A\times B)]\leq \epsilon$ for all $\nu\in \MM$. 

\ref{trivia:ii}. Set $\nu':=\rho\star \nu$. % Fix a martingale transport $\rho$ of $\mu_1'$ to $\mu_1$. For $\nu\in \MM$ define a measure $\nu':=\rho\star\nu$ on $\mathcal{B}_{\JJ^2}$ as follows: if
Let also $(\nu_x)_{x\in\JJ}$ be a disintegration of $\nu$ against $\mu_1$ and $(\rho_{x'})_{x'\in \JJ}$ be a disintegration of $\rho$ relative to $\mu_1'$. Then by definition $\nu'(dx',dy)=\mu_1'(dx')\int_x \rho_{x'}(dx)\nu_x(dy)$. The relation $J(\nu')\geq J(\nu)$ then follows by Jensen's inequality, using the concavity of $\phi$, the convexity of $\gamma$, and the nondecreasingness of $\phi$. Indeed,
\begin{align*}
  J(\nu') & =\int \mu_1'(dx')\phi\left(\int \rho_{x'}(dx)\int \nu_x(dy)\gamma(y)-\gamma(x')\right)\\
          & \geq \int \mu_1'(dx')\phi\left(\int \rho_{x'}(dx)\int \nu_x(dy)\gamma(y)-\int \rho_{x'}(dx)\gamma(x)\right) \\
          & \geq \int \mu_1'(dx')\int \rho_{x'}(dx)\phi\left(\int\nu_x(dy)\gamma(y)-\gamma(x)\right) \\
          & =\int  \mu_1(dx)\phi\left(\int\nu_x(dy)\gamma(y)-\gamma(x)\right)\\
          & =J(\nu).
\end{align*}

\ref{trivia:ii'}. Set $\nu':=\nu\star \rho$. % Fix a martingale transport $\rho$ of $\mu_1'$ to $\mu_1$. For $\nu\in \MM$ define a measure $\nu':=\rho\star\nu$ on $\mathcal{B}_{\JJ^2}$ as follows: if
Let also $(\nu_x)_{x\in\JJ}$ be a disintegration of $\nu$ against $\mu_1$ and $(\rho_{y})_{y\in \JJ}$ be a disintegration of $\rho$ relative to $\mu_2$. Then $\nu'(dx,dy')=\mu_1(dx)\int_y \rho_{y}(dy')\nu_x(dy)$ and we have
\begin{align*}
  J(\nu') & =\int \mu_1(dx)\phi\left(\int \nu_x(dy)\int \rho_{y}(dy')\gamma(y')-\gamma(x)\right)\\
          & \geq\int \mu_1(dx)\phi\left(\int \nu_x(dy)\gamma(y)-\gamma(x)\right) \\
          & =J(\nu).
\end{align*}
The claim follows.

\ref{trivia:iii}. For $s\in S$ set $p_s:=\mu_1(\{s\})$; we may assume $p_s>0$ for all $s\in S$. Let $(\nu^n)_{n\in \mathbb{N}_0}$ be a sequence in $\MM$ and assume $\nu^n\to \nu^0$ weakly as $n\to\infty$. Write $\nu^n=:\sum_{s\in \supp(\mu_1)}p_s\delta_{s}\times \nu_s^n$ for $n\in \mathbb{N}_0$.   For each $s\in \supp(\mu_1)$,  let $f_s$ be any continuous bounded function on $\JJ$ such that $f_s(s')=\delta_{ss'}$ for all $s'\in S$ (it exists because $S$ has no limit points in $\JJ$). Then, noting that $\gamma$ is continuous, as $n\to\infty$, $p_s\nu^n_s[\gamma]=\nu^n[f_s\otimes \gamma]\to \nu^0[f_s\otimes \gamma]=p_s\nu^0_s[\gamma]$, where the fact that $\gamma$ is not necessarily bounded can be handled by a truncation, exploiting the fact that all of the $\nu^n$, $n\in \mathbb{N}$, have the same second marginal that integrates $\gamma$: let $M\in [0,\infty)$; then for $n\in \mathbb{N}$,
\begin{align*}
  \vert \nu^n[f_s\otimes \gamma]  -\nu^0[f_s\otimes \gamma]\vert 
        & \leq \vert \nu^n[f_s\otimes \gamma]-\nu^n[f_s\otimes ((\gamma\lor (-M))\land M)]\vert \\
        & \quad +\vert \nu^n[f_s\otimes ((\gamma\lor (-M))\land M)]-\nu^0[f_s\otimes ((\gamma\lor (-M))\land M)]\vert\\
        & \quad +\vert \nu^0[f_s\otimes ((\gamma\lor (-M))\land M)]-\nu^0[f_s\otimes \gamma]\vert\\
        & \leq 2\Vert f_s\Vert_\infty \mu_2[\vert\gamma-((\gamma\lor (-M))\land M)\vert ]\\
        & \quad +\vert \nu^n[f_s\otimes ((\gamma\lor (-M))\land M)]-\nu^0[f_s\otimes ((\gamma\lor (-M))\land M)]\vert.
\end{align*}
Now let $n\to\infty$ and $M\to\infty$ (in this order). Therefore, for $n\in \mathbb{N}$, $J(\nu^n)=\sum_{s\in S}p_s\phi(\nu_s^n[\gamma]-\gamma(s))\to J(\nu_0)$ as $n\to \infty$, by the continuity of $\phi$ (when $S$ is denumerable the convergence is justified by bounded convergence using the boundedness of $\phi$). %But $\phi$ is concave, continuous and nonnegative, which renders $\alpha\phi(x)\leq \phi(\alpha x)$ for $\alpha\in [0,1]$ and $x\in [0,\infty)$. Therefore $p_s\phi(\nu_s^n[\gamma]-\gamma(s))\leq \phi(p_s\nu_s^n[\gamma]-p_s\gamma(s))\leq \phi
\end{proof}

\begin{question}
Is there always an optimal point in \eqref{eq:opt}? In particular, is the functional $J$ always weakly (upper semi-) continuous on $\MM$? Partial answers to the first question will be given in Theorem~\ref{theorem:general} and Proposition~\ref{proposition:duality}. The answer to the second question is likely to the negative because conditional expectations have a very delicate behavior under weak convergence. The following example demonstrating this phenomenon is due to J. Warren (private communication).
%Since by Prohorov's theorem $\MM$ is weakly compact, it would be true if $J$ is upper semicontinuous in the weak topology. Is this the case?
\end{question}

\begin{example}
 Let $\JJ=(0,1)$. For $n\in \mathbb{N}$ let $\PP_n=2\mathbbm{1}_{P_n}\cdot \mathfrak{l}^2$, where $\mathfrak{l}^2$ is Lebesgue measure on $\mathcal{B}_{(0,1)^2}$ and $P_n:=\left(\cup_{k=1,k\text{ odd}}^{n}(\frac{k-1}{n},\frac{k}{n})\times (0,\frac{1}{2})\right)\cup \left(\cup_{k=1,k\text{ even}}^{n}(\frac{k-1}{n},\frac{k}{n})\times (\frac{1}{2},1)\right)$. Then $\PP_n\to \mathfrak{l}^2$ weakly as $n\to \infty$. However it is not the case that one would have $\PP_n[(\PP_n[\mathbbm{1}_{(0,\frac{1}{2})}(X_2)\vert X_1])^2]\to \mathfrak{l}^2[(\mathfrak{l}^2[\mathbbm{1}_{(0,\frac{1}{2})}(X_2)\vert X_1])^2]$ as $n\to \infty$. In fact $\PP_n[(\PP_n[\mathbbm{1}_{(0,\frac{1}{2})}(X_2)\vert X_1])^2]=\frac{1}{2}$ for all $n\in \mathbb{N}$, while $\mathfrak{l}^2[(\mathfrak{l}^2[\mathbbm{1}_{(0,\frac{1}{2})}(X_2)\vert X_1])^2]=0$. Besides, clearly one can replace the square and $\mathbbm{1}_{(0,\frac{1}{2})}$ in the preceding with suitable bounded continuous functions and still the convergence  $\PP_n[(\PP_n[\mathbbm{1}_{(0,\frac{1}{2})}(X_2)\vert X_1])^2]\to \mathfrak{l}^2[(\mathfrak{l}^2[\mathbbm{1}_{(0,\frac{1}{2})}(X_2)\vert X_1])^2]$ as $n\to\infty$ will fail.% Of course this is not in itself an example demonstrating that the functional $J$ is not in general weakly (upper semi-) continuous, but it does indicate that proving the latter is not likely to be a straightforward exercise (if it can be proved at all).
\end{example}
The following result plays the r\^ole of \cite[Theorem~5.2]{guyon} in our more general setting. It identifies a canonical upper bound for $\sup_{\nu\in \mathcal{M}}J(\nu)$ and characterizes (under fairly innocuous assumptions on $\phi$) when this upper bound is attained.
 
\begin{proposition}\label{proposition:good_scenario}
We have that $\sup_{\nu\in \mathcal{M}}J(\nu)\leq \phi(\mu_2[\gamma]-\mu_1[\gamma])$. Moreover,  provided $\phi$ is injective and strictly concave, then for a $\nu\in \mathcal{M}$, the following are equivalent:
\begin{enumerate}[(i)]
\item\label{good:i} $V_\nu$ is constant a.s.-$\nu$.
\item\label{good:ii} $\nu$ is optimal for \eqref{eq:opt} and the optimal value is equal to $\phi(\mu_2[\gamma]-\mu_1[\gamma])$.
\item\label{good:iii} ``The $\gamma$-increment of $(X_1,X_2)$ is uncorrelated with $X_1$ under $\nu$'', that is to say: $\nu\left[\gamma(X_2)-\gamma(X_1)\vert X_1\right]=\mu_2[\gamma]-\mu_1[\gamma]$ a.s.-$\nu$.
\item\label{good:iv} $(\gamma(X_1)-\mu_1[\gamma],\gamma(X_2)-\mu_2[\gamma])$ is a martingale under $\nu$.
\end{enumerate}
%When it is so, then $V_\nu=J(\nu)=\phi(\mu_2[\gamma]-\mu_1[\gamma])$ a.s.-$\nu$, i.e. $\nu[\gamma(X_2)\vert X_1]=\gamma(X_1)+\mu_2[\gamma]-\mu_1[\gamma]$ a.s.-$\nu$.
\end{proposition}
\begin{remark}
The equivalent conditions of the proposition certainly hold when $\gamma(X_2)-\gamma(X_1)$ is independent of $X_1$ under $\nu$.% (this being in particular the case in any exponential martingale model generated by a process with independent increments). 
\end{remark}
\begin{proof}
  Let $\nu \in \mathcal{M}$. Then by Jensen's inequality
  \begin{align*}
    J(\nu) & = \nu \left[\phi(\nu[\gamma(X_2)-\gamma(X_1)\vert X_1])\right] \\
           & \leq \phi(\nu \left[\nu[\gamma(X_2)-\gamma(X_1)\vert X_1]\right])\\
           & =\phi(\nu[\gamma(X_2)-\gamma(X_1)])=\phi(\mu_2[\gamma]-\mu_1[\gamma]).
  \end{align*}
  Assume now $\phi$ is injective and strictly concave. Then in the preceding inequality  there is equality only if $V_\nu$ is constant a.s.-$\nu$; hence, \ref{good:ii} implies \ref{good:i}. Clearly \ref{good:i}, \ref{good:iii} and \ref{good:iv} are equivalent. Finally, suppose that \ref{good:i} holds true. Then  $\nu[V_\nu]=\phi(\nu[\phi^{-1}(V_\nu)])= \phi(\mu_2[\gamma]-\mu_1[\gamma])$, and \ref{good:ii} follows. 
%The last statement of the proposition is immediate.
\end{proof}
\begin{example}
The equivalent conditions of Proposition~\ref{proposition:good_scenario} may fail for all $\nu\in \MM$. For instance, when $\JJ=(0,\infty)$ and $\gamma=-\frac{2}{\tau}\ln$ for a $\tau\in (0,\infty)$, let $d\in (0,1]$, $u\in [1,\infty)$, $U\in [u,\infty)$ and $D\in (0,d]$. A simple consideration reveals that there exists a unique probability measure $\PP$ on $\mathcal{B}_{\JJ^2}$ supported by the set $\{d,u\}\times \{U,D\}$ and rendering $(X_1,X_2)$ a unit-mean martingale. Let $\mu_1$ and $\mu_2$ be the first and second marginal of $\PP$, respectively. Then $\MM=\{\PP\}$. However, it is easy to check that $\PP[\ln(\frac{X_2}{X_1})\vert X_1=u]=\PP[\ln(\frac{X_2}{X_1})\vert X_1=d]$ iff $u=d$ ($=1$), which of course may fail to be the case.
\end{example}
\begin{example}\label{example}
  The equivalent conditions of Proposition~\ref{proposition:good_scenario} may be satisfied by more than one $\nu\in \mathcal{M}$; in particular there may be more than one (and indeed infinitely many) optimizers in \eqref{eq:opt}. This may be seen, again when $\JJ=(0,\infty)$ and $\gamma=-\frac{2}{\tau}\ln$ for a $\tau\in (0,\infty)$, by considering a situation in which the support of $\mu_1$ consists of two points, while the support of $\mu_2$ consists of four points \emph{and} there is a $\nu\in \MM$ satisfying the equivalent conditions of the previous proposition, with the support of $\nu_i:=\nu(X_2\in \cdot,X_1=i)/\mu_1(\{i\})$ being equal to the support of $\mu_2$ for all $i\in \supp(\mu_1)$ (such measures \emph{do} exist; we give a concrete example below). Then given $\nu\in \MM$, we can characterise $\nu$ by the $8$ real parameters $\nu_i(\{j\})$ corresponding to the disintegration of $\nu$ against $\mu_1$, where $i\in \supp(\mu_1)$, $j\in \supp(\mu_2)$, and these 8 parameters are subject to (at most) $7$ independent linear constraints if we include the condition originating from \ref{good:i} above: 
  \begin{itemize}
  \item $\sum_{i}\mu_1(\{i\})\nu_i(\{j\})=\mu_2(\{j\})$ for all $j\in \supp(\mu_2)$ (4 constraints)
  \item $\sum_j \nu_i(\{j\})=1$ for one $i\in\supp(\mu_1)$ (and the other $i \in \supp(\mu_1)$ follows; 1 constraint)
  \item $\sum_j j\nu_i(\{j\})=i$ for one $i\in\supp(\mu_1)$ (and the other $i \in \supp(\mu_1)$ follows; 1 constraint)
  \item the constraint in \ref{good:i} (1 constraint).
  \end{itemize}
  In addition, the parameters must also be nonegative, $\nu_i(\{j\})\geq 0$.

  If there is some solution to the linear constraints that satisfies the inequalities strictly, by the rank-nullity theorem and continuity there are in fact infinitely many solutions to the linear constraints that satisfy also the inequalities. We show that this can be the case with a concrete example:

 It will suffice to find real numbers $0<a_1 < a_2$ (the atoms of $\mu_1$),  $0<b_1 < b_2 < b_3 < b_4$ (the atoms of $\mu_2$), and $p_i, q_i>0$, $i\in \{1,2,3,4\}$ (the conditional probabilities out of $a_1$ and $a_2$), such that 
 \begin{align*}
   1 & =  p_1+p_2+p_3+p_4, \quad & a_1=p_1b_1+p_2b_2+p_3b_3+p_4b_4,\\
   1 & =  q_1+q_2+q_3+q_4, \quad & a_2=q_1b_1+q_2b_2+q_3b_3+q_4b_4,
 \end{align*}
 and
 \begin{align*}
   \sum_{i =1}^4 p_i \ln(b_i) - \ln(a_1) & = \sum_{i=1}^4 q_i \ln(b_i) -\ln(a_2).
 \end{align*}
 Then considering $b_i=:\alpha_i b_1$, $i\in \{2,3,4\}$, and eliminating $p_1$, $q_1$, $a_1$ and $a_2$, it will suffice to find real numbers $0<\alpha_2,\alpha_3,\alpha_4$ distinct and not equal to $1$, and $p_i>0$, $q_i>0$, $i\in \{2,3,4\}$, with $p_2+p_3+p_4<1$, $q_2+q_3+q_4<1$, such that $$\alpha_2^{p_2-q_2}\alpha_3^{p_3-q_3}\alpha_4^{p_4-q_4}=\frac{1+p_2(\alpha_2-1)+p_3(\alpha_3-1)+p_4(\alpha_4-1)}{1+q_2(\alpha_2-1)+q_3(\alpha_3-1)+q_4(\alpha_4-1)}$$ and $p_2(\alpha_2-1)+p_3(\alpha_3-1)+p_4(\alpha_4-1)\ne q_2(\alpha_2-1)+q_3(\alpha_3-1)+q_4(\alpha_4-1)$.  It is not obvious, but this can be done. For instance with $p_2 = 0.2$, $p_3 = 0.3$, $p_4 = 0.4$, $q_2 =0.4$, $q_3 = 0.3$, $q_4 = 0.2$, $\alpha_2 = 2$, $\alpha_3 = 3.5$, solving numerically gives $\alpha_4\doteq 4.61$ (of course, the existence of such an $\alpha_4$ could be argued analytically in a straightforward, albeit tedious fashion).  \end{example}

\subsection{Case when $\mu_1$ is supported on two points}\label{section:two-point}

Suppose $\supp(\mu_1)=\{a_1,a_2\}\subset \JJ$ with $a_1<a_2$. %Recall we assume $\mu_1$ and $\mu_2$ are increasing in the convex order and that $\phi$ is strictly concave. 
Denote $p_1:=\mu_1(\{a_1\})$ and $p_2:=\mu_1(\{a_2\})$. 

Following on from Meta-corollary~\ref{remark:fundamental}, and simplifying slightly further, we see that our optimization problem (up to identifying an optimizer and hence the optimal value) can be recast in the following form:
\begin{equation}\label{eq:eqivalent_form}
\max_{x\in \mathcal{Y}}\left( p_1\phi\left(x-\gamma(a_1)\right)+p_2\phi\left(\frac{\mu_2[\gamma]-p_1x}{p_2}-\gamma(a_2)\right)\right),
\end{equation}
where $\mathcal{Y}:=\{\mu[\gamma]:\mu\text{ a probability on $\mathcal{B}_{\JJ}$ such that }\mu[\id]=a_1\text{ and }p_1\mu\leq \mu_2\}$ is a compact interval of $\mathbb{R}$ of the form $[x_*,x^*]$ with $x_*=\nu_*[\gamma(X_2);X_1=a_1]/\mu_1(\{a_1\})$ and $x^*=\nu^*[\gamma(X_2);X_1=a_1]/\mu_1(\{a_1\})$ corresponding to the two (possibly one, if they coincide) curtain martingale transports $\nu_*$, $\nu^*$ from $\mu_1$ to $\mu_2$. Specifically,  given an optimal $\hat{x}$ for \eqref{eq:eqivalent_form}, an optimizer for \eqref{eq:opt} is $\nu*:=\lambda \nu_*+(1-\lambda)\nu^*$, where $\lambda\in [0,1]$ is such that $\hat{x}=\lambda x_*+(1-\lambda)x^*$.

%To carry the above analysis one step further, assume for simplicity up to the end of this subsection 
%\begin{condition}\label{assumption}
%$\phi$ is strictly concave.
%\end{condition}

Now,  because a positive combination of (strictly) concave functions is (strictly) concave, and because (strict) concavity is not affected by precomposition with a (non-constant) affine function, we see that the objective functional in \eqref{eq:eqivalent_form}  is in fact (strictly) concave on its natural domain $[\gamma(a_1),\gamma(a_1)+\frac{\mu_2[\gamma]-\mu_1[\gamma]}{p_1}]=:D$ (provided $\phi$ is strictly concave). In particular it means that there is only one maximizer $\hat{x}$ when $\phi$ is strictly concave. Furthermore, $D\supset [x_*,x^*]$ (which fact is automatic, because otherwise \eqref{eq:eqivalent_form} would not be well-defined) and $ D\ni x_0:=\gamma(a_1)+\mu_2[\gamma]-\mu_1[\gamma]=p_2(\gamma(a_1)-\gamma(a_2))+\mu_2[\gamma]$. Finally, at $x=x_0$, by Jensen's inequality, the objective functional in \eqref{eq:eqivalent_form} attains its largest value, $\phi(\mu_2[\gamma]-\mu_1[\gamma])$, on $D$.  

Therefore: if $x_0\in [x_*,x^*]$, then we may take $\hat{x}=x_0$ and, assuming further that $\phi$ is injective \& strictly concave, this corresponds to the situation described by Proposition~\ref{proposition:good_scenario}; if $x_0<x_*$, then we may take $\hat{x}=x_*$; finally if $x_0>x^*$, then we may take $\hat{x}=x^*$. The latter two cases correspond to $\nu*$ being one of the curtain martingale transports. To summarize, we may take $\hat{x}=(x_0\lor x_*)\land x^*$.

The above then constitutes a complete analytic solution to \eqref{eq:opt}, at least as far as finding \emph{an} optimizer is concerned (and hence automatically the corresponding optimal value), in the case when the support of $\mu_1$ is a two-point set. 

\begin{remark}\label{remark:min}
The preceding is modified in a straightforward manner to handle the case when $\min$ replaces $\max$ in \eqref{eq:opt}: simply replace $\max$ by $\min$ in \eqref{eq:eqivalent_form}. In that case there is always an optimizer for \eqref{eq:eqivalent_form} on the boundary of $\mathcal{Y}$ (because the objective functional in \eqref{eq:eqivalent_form} is concave) and any minimizer is necessarily on the boundary of $\mathcal{Y}$ if $\phi$ is even strictly concave (because then the objective functional in \eqref{eq:eqivalent_form} too is even strictly concave). 
\end{remark}

\subsection{Case when the support of $\mu_1$ is finite}\label{subsection:finite-support}
Suppose now $\mu_1$ is supported by the finite set $\{a_1,\ldots,a_n\}$ consisting of $n\in \mathbb{N}$ elements. (Of course the case $n=1$ is trivial, while the case $n=2$ was treated in the previous subsection, so the following is only interesting for $n\geq 3$.) Set $p_i:=\lambda_1(\{a_i\})$ for $i\in \{1,\ldots,n\}$.

We may write the objective functional in \eqref{eq:opt} as, for $\nu\in \mathcal{M}$, $$J(\nu)=\sum_{i=1}^np_i\phi(\nu_i[\gamma]-\gamma(a_i)),$$ where $\nu_i:=\nu(X_2\in \cdot\vert X_1=a_i)$ for $i\in \{1,\ldots,n\}$.

We see then that we are precisely in the setting of Meta-corollary~\ref{remark:fundamental} and hence the procedure for finding an optimizer to \eqref{eq:opt} described there applies. Specifically, the associated Euclidean space problem is now 

\begin{equation}\label{eq:finite-equivalent}
\max_{x\in \mathcal{X}}\sum_{i=1}^np_i\phi(x_i-\gamma(a_i)),
\end{equation}
where $\mathcal{X}:=\{(\nu_1[\gamma],\ldots,\nu_n[\gamma]):\nu\in \mathcal{M}\}$.

Unlike when $n=2$ it is no longer possible to give ``nice'' closed-form expressions for an optimizer. %But, numerically, finding an optimizer using Meta-corollary~\ref{remark:fundamental} is straightforward, at least for sufficiently small $n$. 

\begin{remark}\label{remark:min-finite}
The preceding discussion also holds if $\min$ replaces $\max$ in \eqref{eq:opt}. However the  minimisation problem is  then seen to be less interesting than the maximisation problem because in that case an optimal point is to be found in a vertex of $\mathcal{X}$ by essentially the same argument as in the case when the support of $\mu_1$ consisted of two points. We leave the details to the reader.
\end{remark}
\begin{remark}\label{remark:qualitative-finite}
The natural domain of the objective functional in \eqref{eq:finite-equivalent} is $\prod_{i=1}^n[\gamma(a_i),\infty)=:H$; of course $ H\supset \mathcal{X}$. In addition, $x_0:=(\gamma(a_i)+\mu_2[\gamma]-\mu_1[\gamma])_{i=1}^n\in H$, and by the concavity of $\phi$, the objective functional of \eqref{eq:finite-equivalent} attains its highest value on $H$ at $x_0$. Therefore, if $x_0\in \mathcal{X}$, then $x_0$ is optimal for \eqref{eq:finite-equivalent} and the corresponding optimal martingale transport $\hat{\nu}$ renders $V_{\hat{\nu}}$ constant on the support of $\mu_1$. If in addition $\phi$ is strictly concave, then $x_0$ is the only optimizer for  \eqref{eq:finite-equivalent}; if, moreover, $\phi$ is also injective, then this corresponds to the situation described by Proposition~\ref{proposition:good_scenario}. Conversely, if $x_0\notin \mathcal{X}$, then again by the concavity of $\phi$, a maximizer of \eqref{eq:finite-equivalent} can be found on the boundary $\partial \mathcal{X}$ of $\mathcal{X}$; if $\phi$ is even strictly concave, then any maximizer of  \eqref{eq:finite-equivalent} lies in $\partial \mathcal{X}$. 
\end{remark}
%For the remainder of this subsection we assume again that Condition~\ref{assumption} is in effect. 

\subsection{General case}\label{subsection:general}
In this subsection we show a ``continuity'' statement for \eqref{eq:opt} in the first marginal $\mu_1$, which allows us effectively to reduce the general case to the case considered in Subsection~\ref{subsection:finite-support}.
\begin{theorem}\label{theorem:general}
Assume  that: 
\begin{enumerate}[(a)]
%\item $\mu_1$ is diffuse.
\item\label{gamma} $\gamma$ is  bounded; and %and locally Lipschitz continuous; and 
\item\label{phi} $\phi\vert_{[0,2\Vert\gamma\Vert_\infty]}$ is nondecreasing and continuous, while $\phi\vert_{(0,2\Vert\gamma\Vert_\infty]}$ is locally Lipschitz.
\end{enumerate}
Then there exists a sequence of finitely supported measures on $\mathcal{B}_\JJ$, $(\mu^n_1)_{n \in \mathbb{N}}$, nondecreasing in convex order, converging weakly to $\mu_1$, and in convex order with respect to $\mu_1$, such that, if, for $n\in \mathbb{N}$, $\nu^n$ is a maximiser for \eqref{eq:opt} with $\mu^n_1$ replacing $\mu_1$, then:
\begin{enumerate}[(I)]
\item\label{thm:I} $\downarrow\!\!\text{-}\lim_{n\to\infty}J(\nu^n)=\max_{\nu\in \MM}J(\nu)$. 
\item\label{thm:II} The sequence $(\nu^n)_{n\in \mathbb{N}}$ is tight. 
\item\label{thm:III} Any weak accumulation point $\nu^0$ of the sequence $(\nu^n)_{n\in \mathbb{N}}$ is a maximizer for \eqref{eq:opt}, i.e. $\nu^0\in \MM$ and  $J(\nu^0)=\max_{\nu\in \MM}J(\nu)$.
\end{enumerate}
Such a sequence of maximizers $(\nu^n)_{n\in \mathbb{N}}$ and an associated accumulation point $\nu^0$ exist.
\end{theorem}

\begin{remark}
\leavevmode
  \begin{enumerate}[(i)]
  \item The proof will in fact provide a simple recipe for constructing the sequence of approximating measures $\mu^n_1$ in terms of the original measure $\mu_1$. This approximation does not depend on $\mu_2$.
  \item For the proof technique that we use, the assumption that $\gamma$ is bounded (rather than, say, just locally bounded) appears to be crucial.
  \item If $\mu_1$ and $\mu_2$ have supports that are compactly contained in $\JJ$, then we may pass from $\JJ$ to an open subinterval $\JJ$ that is itself compactly contained in $\JJ$, and condition \ref{gamma} is met. 
  \end{enumerate}
\end{remark}
\begin{proof}
Let $(\LL_n)_{n\in \mathbb{N}}$ be a sequence of finite partitions of $\JJ$ such that:
\begin{enumerate}[(i)]
\item\label{partition:iii} for each $n\in \mathbb{N}$, $\LL_n$ consists of intervals that are continuity sets for $\mu_1$;
\item\label{partition:i} $\LL_{n+1}$ is finer than $\LL_n$ for each $n\in \mathbb{N}$;
\item\label{partition:ii} for each bounded $K\subset \JJ$, $d_n(K):=\max_{I\in \LL_n,I\cap K\ne \emptyset}\mathrm{diam}(I)\to 0$ as $n\to \infty$.
% $\mu_1(\cup\{I\in \LL_n:I\cap [-m,m]\ne \emptyset \text{ and }(\JJ\backslash [-m,m])\cap I\ne \emptyset)\to 0$ as $n\to \infty$, for a sequence of $m$s that $\uparrow\infty$.
\end{enumerate}
Observe that such sequences of partitions certainly exist.

We now define our sequence of approximating measures: for $I\in \mathcal{B}_\JJ$, set $a_I$ equal to $\mu_1[\id\vert I]$ if $\mu_1(I)>0$, and equal to an arbitrary element of $\JJ$ otherwise. Then, for $n\in \mathbb{N}$: define $\mu_1^n:=\sum_{I\in \LL_n} \mu_1(I)\delta_{a_I}$; observe that $\mu^n_1$ is before $\mu_1$ in convex order, and hence also before $\mu_2$.  It is clear that, as $n\to\infty$, weakly and nondecreasingly in convex order, $\mu_1^n\to \mu_1$.  Further, by Proposition~\ref{proposition:trivia}, items \ref{trivia:i} and \ref{trivia:iii}, for each $n\in \mathbb{N}$, we can choose $\nu^n$ to be a maximizer for \eqref{eq:opt} in which $\mu^n_1$ replaces $\mu_1$. We show finally that any such sequence $(\nu^n)_{n\in \mathbb{N}}$ of maximizers has the desired properties \ref{thm:I}-\ref{thm:II}-\ref{thm:III}. Indeed, on account of the monotonicity of Proposition~\ref{proposition:trivia}\ref{trivia:ii} and the fact that the sequence $(\mu^n_1)_{n\in \mathbb{N}}$ is nondecreasing in the convex order, \ref{thm:I} will follow as soon as \ref{thm:II} and \ref{thm:III} are established. 

We focus first on \ref{thm:II}. For any $\epsilon\in (0,\infty)$, there is a compact $K_1\subset \JJ$ such that $\mu_1^n(\JJ\backslash K_1)<\epsilon$ for all sufficiently large natural $n$, while for each $m\in [0,\infty)$ for which $[-m,m]$ is a continuity set of $\mu_1$, $\mu_1^n[\vert \id\vert;\JJ\backslash [-m,m]]\to \mu_1[\vert \id\vert;\JJ\backslash [-m,m]]$ as $n\to\infty$. Then arguments very similar to the ones seen in the proof of Proposition~\ref{proposition:trivia}\ref{trivia:i} will show that the sequence $(\nu^n)_{n\in \mathbb{N}}$ is tight.  Let $\nu^0\in \mathcal{M}$ be any accumulation point of this sequence.

Next we introduce some notation. For $n\in \mathbb{N}$, let $\rho_n:=\sum_{I\in \LL_n}\delta_{a_I}\times \mu_1(\cdot \cap I)$, a martingale transport from $\mu_1^n$ to $\mu_1$. For $m \leq n$, define similarly $\rho_{mn}:=\sum_{I\in \LL_m}\delta_{a_I}\times \mu_1^n(\cdot\cap I)$, a martingale transport of $\mu^m_1$ to $\mu^n_1$. 

Now, by Proposition~\ref{proposition:trivia}\ref{trivia:ii}, for $m\leq n$, we have that $J(\nu^n)\leq J(\rho_{mn}\star \nu ^n)$; hence $\limsup_{n\to \infty}J(\nu^n)\leq \limsup_{n\to \infty} J(\rho_{mn}\star \nu ^n)$. We will show that: 
\begin{enumerate}[(1)]
\item\label{proof:1} $\lim_{n\to\infty}J(\rho_{mn}\star \nu ^n)=J(\rho_m\star \nu^0)$ for each $m\in \mathbb{N}$; and
\item\label{proof:2} $\lim_{m\to \infty}J(\rho_m\star\nu^0)=J(\nu^0)$. 
\end{enumerate}
This will imply $\limsup_{n\to \infty}J(\nu^n)\leq J(\nu^0)\leq \sup_{\nu \in \MM}J(\nu)$. On the other hand, again by Proposition~\ref{proposition:trivia}\ref{trivia:ii}, we will have $\liminf_{n\to \infty}J(\nu^n)\geq \sup_{\nu \in \MM}J(\nu)$, which will render \ref{thm:III} and the proof will be complete.

To prove \ref{proof:2}, let $(\nu^0_x)_{x\in \JJ}$ be a disintegration of $\nu^0$ against $\mu_1$. Let $\varepsilon\in (0,\infty)$. By a classical theorem of Lusin, there exists 
%a continuous function $\omega:\JJ\to [0,\Vert\gamma\Vert_\infty]$ such that $\omega=\nu^0_\cdot[\gamma]$ on a compact set $E$ with $\mu_1(\JJ\backslash E)<\epsilon$. Because $\mu_1$ is diffuse we may also take $\omega=\nu^0_\cdot[\gamma]$ on the (denumerable) set $D:=\cup_{n\in \mathbb{N}}\supp(\mu^n_1)$. Set $K:=E\cup D$. 
a compact set $K\subset \JJ$ with $\mu_1(\JJ\backslash K)<\varepsilon$ such that $K\ni x\mapsto \nu_x^0[\gamma]$ is continuous. In particular, because a continuous function on a compact set is uniformly continuous, we see that $\Delta_m(K):=\max_{I\in \LL_m,I\cap K\ne\emptyset}\sup_{v\in I\cap K,x\in I\cap K}\vert \nu^0_x[\gamma]-\nu^0_v[\gamma]\vert \to 0$ as $m\to \infty$.
%We may choose $K$ so that $\mu_1(\JJ\backslash K^{d_n})<\epsilon$ with $K^{d_n(K)}:=\{k\in \JJ:\mathrm{dist}(k,K)\leq d_n\}$ for all sufficiently large $n\in \mathbb{N}$.
%Making $K$ smaller if necessary we may further ask that $\mu_1^n(\JJ\backslash K)$
Furthermore, we may assume that $\phi(0)=0$. With this assumption having been made, let also $\varepsilon'\in (0,\infty)$ and set $\delta':=\sup\{x\in [0,2\Vert \gamma\Vert_\infty]:\phi(x)\leq \varepsilon'\}$. In particular $\delta'\in (0,2\Vert \gamma\Vert_\infty]$ and $\phi(\delta')\leq \varepsilon'$.

We prepare now the following estimate for $\{a,b\}\subset [0,2\Vert \gamma\Vert_\infty]$ on the function $\phi$. Assume $a\leq b$; then:
\begin{itemize}
 \item if $\{a,b\}\subset [0,\delta')$, then $\vert \phi(b)-\phi(a)\vert=\phi(b)-\phi(a)\leq \phi(b)\leq \phi(\delta')$ (because $\phi$ is nondecreasing, in particular nonnegative);
\item if $a<\delta'\leq b$, then $\vert \phi(b)-\phi(a)\vert= (\phi(\delta')-\phi(a))+(\phi(b)-\phi(\delta'))\leq \phi(\delta')+\Vert \phi\vert_{[\delta',2\Vert \gamma\Vert_\infty]}\Vert_{\mathrm{Lip}}( b-\delta')\leq \phi(\delta')+\Vert \phi\vert_{[\delta',2\Vert \gamma\Vert_\infty]}\Vert_{\mathrm{Lip}}\vert b-a\vert$;
\item finally, if $\{a,b\}\subset [\delta', 2\Vert \gamma\Vert_\infty]$, then $\vert \phi(b)-\phi(a)\vert\leq \Vert \phi\vert_{[\delta',2\Vert \gamma\Vert_\infty]}\Vert_{\mathrm{Lip}}\vert b-a\vert$.
\end{itemize}
So $\vert \phi(b)-\phi(a)\vert\leq  \phi(\delta')+\Vert \phi\vert_{[\delta',2\Vert \gamma\Vert_\infty]}\Vert_{\mathrm{Lip}}\vert b-a\vert$, and the supposition $a\leq b$ may now also be dropped.

Next note that for $m\in \mathbb{N}$: $\rho_m\star \nu^0=\sum_{I\in \LL_m}\delta_{a_I}\times \int_I\nu_x^0(\cdot)\mu_1(dx)$; hence $(\rho_m\star \nu^0)[\gamma(X_2)\vert X_1]=\sum_{I\in  \LL_m}\mathbbm{1}_{\{a_I\}}(X_1)\frac{\int_I\nu_x^0[\gamma]\mu_1(dx)}{\mu_1(I)}$ a.s.-$\rho_m\star\nu^0$; and so %$$(\rho_m\star \nu^0)[\phi((\rho_m\star \nu^0)[\gamma(X_2)\vert X_1]-\gamma(X_1))]%=(\rho_m\star \nu^0)[\phi((\rho_m\star \nu^0)[\gamma(X_2)\vert X_1]-\gamma(X_1));X_1\in K]$$
\begin{align*}
  J&(\rho_m\star \nu^0)=\sum_{I\in \LL_m}\mu_1(I)\phi\left(\frac{\int_I\nu_x^0[\gamma]\mu_1(dx)}{\mu_1(I)}-\gamma(a_I)\right)\\
& =\sum_{I\in \LL_m}\int_{I}\phi\left(\frac{\int_I\nu_x^0[\gamma]\mu_1(dx)}{\mu_1(I)}-\gamma(a_I)\right)\mu_1(dv)\\
& =\sum_{I\in \LL_m}\int_{I\cap K}\phi\left(\frac{\int_I\nu_x^0[\gamma]\mu_1(dx)}{\mu_1(I)}-\gamma(a_I)\right)\mu_1(dv)+\sum_{I\in \LL_m}\int_{I\backslash K}\phi\left(\frac{\int_I\nu_x^0[\gamma]\mu_1(dx)}{\mu_1(I)}-\gamma(a_I)\right)\mu_1(dv).
\end{align*}
On the other hand 
\begin{align*}
J(\nu^0) & =\sum_{I\in \LL_m}\int_{I} \phi(\nu_v^0[\gamma]-\gamma(v))\mu_1(dv)\\
& =\sum_{I\in \LL_m}\int_{I\cap K} \phi(\nu_v^0[\gamma]-\gamma(v))\mu_1(dv)+\sum_{I\in \LL_m}\int_{I\backslash K} \phi(\nu_v^0[\gamma]-\gamma(v))\mu_1(dv).
\end{align*}
Then by the triangle inequality we can estimate (note that $\gamma$, being convex, is locally Lipschitz): 
\begin{align*}
\vert J& (\rho_m\star \nu^0)-J(\nu^0)\vert \\
  & \leq 2\phi(2\Vert\gamma\Vert_\infty]) \mu_1(\JJ\backslash K)\\
   &\qquad {}+\sum_{I\in \LL_m,I\cap K\ne \emptyset}\mu_1(I\cap K)\sup_{v\in I\cap K}\left\vert \phi\left(\frac{\int_I\nu_x^0[\gamma]\mu_1(dx)}{\mu_1(I)}-\gamma(a_I)\right)-\phi(\nu_v^0[\gamma]-\gamma(v))\right\vert\\
& \leq 2\phi(2\Vert\gamma\Vert_\infty]) \mu_1(\JJ\backslash K)+\phi(\delta')\mu_1(K) \\
&\qquad {}+\Vert \phi\vert_{[\delta',2\Vert\gamma\Vert_\infty]}\Vert_{\mathrm{Lip}}\,\mu_1(K)\Vert\gamma\vert_{K+[-d_m(K),d_m(K)]}\Vert_{\mathrm{Lip}}\,d_m(K) \\
& \qquad {} +\Vert \phi\vert_{[\delta',2\Vert\gamma\Vert_\infty]}\Vert_{\mathrm{Lip}}\sum_{I\in \LL_m,I\cap K\ne \emptyset}\sup_{v\in I\cap K}\int_I\vert \nu_x^0[\gamma]-\nu_v^0[\gamma]\vert\mu_1(dx).
\end{align*}
Further, 
\begin{align*}
  \sum_{I\in \LL_m,I\cap K\ne \emptyset} \sup_{v\in I\cap K}\int_I\vert \nu_x^0[\gamma]-&\nu_v^0[\gamma]\vert\mu_1(dx) \\
  & \leq \sum_{I\in \LL_m,I\cap K\ne \emptyset}\int_{I\cap K}\sup_{v\in I\cap K}\vert \nu_x^0[\gamma]-\nu_v^0[\gamma]\vert\mu_1(dx)\\
  &\qquad {}+\sum_{I\in \LL_m,I\cap K\ne\emptyset}\sup_{v\in I\cap K}\int_{I\backslash K} \vert \nu_x^0[\gamma]-\nu_v^0[\gamma]\vert\mu_1(dx)\\
& \leq \mu_1(K)\Delta_m( K)+2\Vert \gamma\Vert_\infty\mu_1(\JJ\backslash K).
\end{align*}
In conclusion 
\begin{align*}
  \vert J(\rho_m\star \nu^0)-J(\nu^0)\vert 
  & \leq 2\phi(2\Vert\gamma\Vert_\infty])\varepsilon+\varepsilon'+\Vert \phi\vert_{[\delta',2\Vert\gamma\Vert_\infty]}\Vert_{\mathrm{Lip}}\Vert\gamma\vert_{K+[-d_m(K),d_m(K)]}\Vert_{\mathrm{Lip}}\,d_m(K)\\ 
  & \qquad {} +\Vert \phi\vert_{[\delta',2\Vert\gamma\Vert_\infty]}\Vert_{\mathrm{Lip}}\,\Delta_m(K)+2\Vert \phi\vert_{[\delta',2\Vert\gamma\Vert_\infty]}\Vert_{\mathrm{Lip}}\Vert \gamma\Vert_\infty\varepsilon.
\end{align*}
Letting $m\to\infty$, $\varepsilon\downarrow 0$ and then $\varepsilon'\downarrow 0$ (in this order) concludes the argument for \ref{proof:2}.

It remains to argue \ref{proof:1}. To this end note that for natural $m\leq n$, one has $\rho_{mn}\star\nu^n=\sum_{I\in \LL_m}\delta_{a_I}\times \nu^n(I\times \cdot)$, while $\rho_m\star \nu^0=\sum_{I\in \LL_m}\delta_ {a_I}\times \nu^0(I\times \cdot)$. Then we may write $$J(\rho_{mn}\star\nu^n)=\sum_{I\in \LL_m}\mu_1(I)\phi\left(\frac{\nu^n[\mathbbm{1}_I\otimes \gamma]}{\mu_1(I)}-\gamma(a_I)\right)$$ and $$J(\rho_m\star \nu^0)=\sum_{I\in \LL_m}\mu_1(I)\phi\left(\frac{\nu^0[\mathbbm{1}_I\otimes \gamma]}{\mu_1(I)}-\gamma(a_I)\right).$$
The desired convergence is now transperent because all the $I\in \LL_m$, are continuity sets of $\mu_1$ by assumption. 
\end{proof}

\subsection{Duality arguments}\label{subsection:duality}
We assume in this section that $\phi$ is nonnegative and then without loss of generality that $\phi(0)=0$. Recalling the notation introduced in Subsection~\ref{sec:class-non-linear}, let $(S_1,S_2,V)$ be the canonical projections on $\JJ^2\times [0,\infty)$ and introduce
$$\mathcal{M}':=\{\text{probability measures $\mu$ on }\mathcal{B}_{\JJ^2\times [0,\infty)}\text{ such that }$$
$${S_1}_\star\mu=\mu_1,\, {S_2}_\star\mu=\mu_2,\, \mu[S_2\vert S_1,V]=S_1\text{ and }\phi(\mu[\gamma(S_2)\vert S_1,V]-\gamma(S_1))=V\text{ a.s.-$\mu$}\}.$$ For a $\mu\in \mathcal{M}'$ define $\mu\vert_{(1,2)}:=(S_1,S_2)_\star \mu$; and for a $\nu\in \mathcal{M}$, define $\nu':=(X_1,X_2,V_\nu)_\star \nu$.

We will consider the optimization problem 
\begin{equation}\label{eq:alternative}
\max_{\mu\in \mathcal{M}'}\mu[V];
\end{equation}
see the Introduction for the motivation behind this. 

The next proposition is a generalization of  \cite[(proof of) Proposition~4.10, Lemma~3.3]{guyon} to our setting. 

\begin{proposition}
We have the following assertions.
\begin{enumerate}[(i)]
\item\label{alternative:i} For a $\mu\in \mathcal{M}'$, $\mu\vert_{(1,2)}\in \mathcal{M}$.
\item\label{alternative:ii} For a $\nu\in \mathcal{M}$, $\nu'\in \mathcal{M}'$.
\item\label{alternative:iii} One has $\sup_{\nu \in \mathcal{M}}\nu[V_\nu]=\sup_{\mu\in \mathcal{M}'}\mu[V]$. Furthermore:
\begin{enumerate}[(a)]
\item\label{alt:a} if $\mu$ is optimal for \eqref{eq:alternative}, then $\mu\vert_{(1,2)}$ is optimal for \eqref{eq:opt}; and conversely,
\item\label{alt:b} if $\nu$ is optimal for \eqref{eq:opt}, then $\nu'$ is optimal for \eqref{eq:alternative}.
\end{enumerate}
\end{enumerate}
\end{proposition}
\begin{proof}
\ref{alternative:i}. By the tower property of conditional expectations $\mu[S_2\vert S_1]=S_1$ a.s.-$\mu$. Hence, because  $\mu\vert_{(1,2)}=(S_1,S_2)_\star \mu$, also $\mu\vert_{(1,2)}[X_2\vert X_1]=X_1$ a.s.-$\mu\vert_{(1,2)}$. 

\ref{alternative:ii}. In this case $\nu[X_2\vert X_1]=X_1$ a.s.-$\nu$ implies $\nu[X_2\vert X_1,V_\nu]=X_1$ a.s.-$\nu$, so that also $\nu'[S_2\vert S_1,V]=S_1$ a.s.-$\nu'$, while $V_\nu=\phi(\nu[\gamma(X_2)\vert X_1]-\gamma(X_1))$ a.s.-$\nu$, renders $\phi(\nu'[\gamma(S_2)\vert S_1,V]-\gamma(S_1))=V$ a.s.-$\nu'$.

\ref{alternative:iii}. Let $\nu\in \mathcal{M}$. Then $\nu'[V]=\nu[V_\nu]=J(\nu)$. Thus, by \ref{alternative:ii}, $\sup_{\mu\in \mathcal{M}'}\mu[V]\geq \sup_{\nu \in \mathcal{M}}J(\nu)$. Conversely, let $\mu\in \mathcal{M}'$. Then by the tower property of conditional expectations, and by conditional Jensen's inequality, exploiting the concavity of $\phi$, 
\begin{align*}
  \mu\vert_{(1,2)}[V_{\mu\vert_{(1,2)}}]
  & =\mu\vert_{(1,2)}[\phi(\mu\vert_{(1,2)}[\gamma(X_2)\vert X_1]-\gamma(X_1))]\\
  & =\mu[\phi(\mu[\gamma(S_2)\vert S_1]-\gamma(S_1))]\\
  & =\mu[\phi(\mu[\mu[\gamma(S_2)\vert S_1,V]-\gamma(S_1)\vert S_1])]\\
  & \geq \mu[\mu[\phi(\mu[\gamma(S_2)\vert S_1,V]-\gamma(S_1))\vert S_1]]\\
  & =\mu[\phi(\mu[\gamma(S_2)\vert S_1,V]-\gamma(S_1))]=\mu[V].
\end{align*}
Hence, from \ref{alternative:i}, $\sup_{\mu\in \mathcal{M}'}\mu[V]\leq \sup_{\nu \in \mathcal{M}}J(\nu)$. By the preceding also \ref{alt:a} and \ref{alt:b} follow.
\end{proof}
%Assume now $\phi$ is injective. The dual problem comes about from the following manipulations:
%
%$$\sup_{\mu\in \mathcal{M}'} \mu[V]=$$\footnotesize
%$$\sup_{\mu}\inf_{(u_1,u_2,\Delta,\Gamma)}\left\{\mu_1[u_1]+\mu_2[u_2]-\mu[(u_1(S_1)+u_2(S_2)+\Delta(S_1,V)(S_2-S_1)+\Gamma(S_1,V)(\gamma(S_2)-\gamma(S_1)-\phi^{-1}(V))-V)^+]\right\}$$\normalsize
%(with $u_1\in \mathcal{L}^1(\mu_1)$, $u_2\in \mathcal{L}^1(\mu_2)$, $\{\Delta,\Gamma\}\subset \mathcal{B}_{\JJ\times [0,\infty)}/\mathcal{B}_\mathbb{R}$, and $\mu$ a measure on $\mathcal{B}_{\JJ^2\times [0,\infty)}$), which is certainly $\leq$ \footnotesize
%$$\inf_{(u_1,u_2,\Delta,\Gamma)}\sup_{\mu}\left\{\mu_1[u_1]+\mu_2[u_2]-\mu[(u_1(S_1)+u_2(S_2)+\Delta(S_1,V)(S_2-S_1)+\Gamma(S_1,V)(\gamma(S_2)-\gamma(S_1)-\phi^{-1}(V))-V)^+]\right\}$$\normalsize
%$$=\inf_{(u_1,u_2,\Delta,\Gamma)\in \mathcal{U}}\mu_1[u_1]+\mu_2[u_2],$$
%Assume now $\phi$ is injective, continuous, and define 
%$$\mathcal{U}':=\{(u_1,u_2,\Delta,\Gamma)\in \mathcal{L}^1(\mu_1)\times \mathcal{L}^1(\mu_2)\times (b\mathcal{B}_{\JJ\times [0,\infty)})^2:$$
%$$u_1(S_1)+u_2(S_2)+\Delta(S_1,V)(S_2-S_1)+\Gamma(S_1,V)(\gamma(S_2)-\gamma(S_1)-\phi^{-1}(V))\geq V\}.$$

Moreover, corresponding to \cite[Theorem~4.1]{guyon}, we have the following duality result. Before we state it, we assume henceforth in this subsection that $\phi$ is injective (i.e. strictly increasing), continuous and that $\lim_\infty\phi=\infty$. This means that $\phi^{-1}$ is well-defined, strictly convex, strictly increasing, and, like $\phi$, maps $[0,\infty)$ onto $[0,\infty)$. We set

$$\mathcal{U}':=\{(u_1,u_2,\Delta,\Gamma)\in \mathcal{L}^1(\mu_1)\times \mathcal{L}^1(\mu_2)\times b\mathcal{B}_{\JJ\times [0,\infty)}\times (\mathcal{B}_{\JJ\times [0,\infty)}/\mathcal{B}_\mathbb{R}):$$
$$u_1(S_1)+u_2(S_2)+\Delta(S_1,V)(S_2-S_1)+\Gamma(S_1,V)(\gamma(S_2)-\gamma(S_1)-\phi^{-1}(V))\geq V\}.$$

\begin{theorem}\label{proposition:duality}
Assume that $\lim_{(\inf\JJ) +}(\vert\gamma\vert+\vert \mathrm{id}\vert)=\lim_{(\sup\JJ) -}(\vert\gamma\vert+\vert \mathrm{id}\vert)=\infty$. There is the following duality: 
\begin{equation}\label{duality}
\sup_{\mu\in \mathcal{M}'} \mu[V]=\inf_{(u_1,u_2,\Delta,\Gamma)\in \mathcal{U}'}\mu_1[u_1]+\mu_2[u_2].
\end{equation}
Furthermore, the supremum in \eqref{duality} is attained, and in the definition of $\mathcal{U}'$ we could have insisted, without affecting the validity of \eqref{duality}, that the functions $u_1,u_2$ are each the difference of two convex functions, while the functions $\Delta,\Gamma$ are continuous.
\end{theorem}
\begin{proof}
The inequality ``$\leq$'' in \eqref{duality} follows from the definitions of $\mathcal{M}'$ and $\mathcal{U}'$. Indeed let $\mu\in \MM'$ and $(u_1,u_2,\Delta,\Gamma)\in \UU'$. Then $\Gamma(S_1,V)(\gamma(S_2)-\gamma(S_1)-\phi^{-1}(V))\geq V-u_1(S_1)-u_2(S_2)-\Delta(S_1,V)(S_2-S_1)$, with the right-hand side and therefore the left-hand side having a $\mu$-integrable negative part. Taking $\mu$-expectations yields the claim.\footnote{If $\PP$ is a probability measure on $(\Omega,\FF)$, $\GG$ a sub-$\sigma$-field of $\FF$, and $X\in \mathcal{L}^1(\PP)$ is such that $\PP[X;G]=0$ for all $G\in \GG$, then $\PP[Xg]=0$ for all $g\in \GG/\mathcal{B}_\mathbb{R}$ for which $\PP[(Xg)^+]\land \PP[(Xg)^-]<\infty$: If $g$ is bounded then this is well-known. Let $g\in \GG/\mathcal{B}_{[0,\infty)}$. Then, for $n\in \mathbb{N}$, $0=\PP[X(g\land n)]=\PP[X^+(g\land n)]-\PP[X^-(g\land n)]\to \PP[X^+g]-\PP[X^-g]=\PP[Xg]$ as $n\to \infty$, by monotone convergence and because by assumption $\PP[X^+g]\land \PP[X^-g]<\infty$. In the general case $\PP[(Xg)^+]\land \PP[(Xg)^-]<\infty$ implies $\PP[(Xg^\pm)^+]\land \PP[(Xg^\pm)^-]<\infty$, and hence by what we have just shown $\PP[Xg]=\PP[Xg^+]-\PP[Xg^-]=0$.}

For the reverse inequality we follow closely \cite[Section~4]{guyon}.

First, by a classical theorem of de la Vall\'ee-Poussin, applied to the probability measure $\mu_1\times \mu_2$ and the $\mu_1\times \mu_2$-uniformly integrable (as finite integrable) family $\{X_1\otimes 1,1\otimes X_2,\gamma(X_1)\otimes 1,1\otimes \gamma(X_2)\}$, there exists a convex, strictly increasing (hence continuous) %and differentiable
$\xi: [0,\infty)\to [0,\infty)$ of superlinear growth with $\xi(0)=\xi'(0+)=0$, such that $\mu_i[\xi(\vert  \cdot\vert)]<\infty$ and $\mu_i[\xi(\vert \gamma\vert)]<\infty$ for $i\in \{1,2\}$. 

Second, set $\theta(y):=\inf_{b\in (0,\infty)}\frac{y+\xi( b/2)}{b}$ for $y\in [0,\infty)$. Then $\theta:[0,\infty)\to [0,\infty)$ is concave, $\lim_\infty\theta=\infty$, and $\theta(y)b\leq y+\xi(\vert b\vert/2)$ for all $(y,b)\in [0,\infty)\times \mathbb{R}$. Therefore, for $(s_1,s_2,v)\in \JJ^2\times [0,\infty)$, choosing $y=\phi^{-1}(v)$ and $b=\gamma(s_2)-\gamma(s_1)$ in the preceding, we obtain 
$$\theta(\phi^{-1}(v))(\gamma(s_2)-\gamma(s_1))\leq \phi^{-1}(v)+\xi(\vert \gamma(s_2)-\gamma(s_1)\vert/2)\leq \phi^{-1}(v)+\xi(\vert \gamma(s_1)\vert)+\xi(\vert \gamma(s_2)\vert),$$ 
i.e. 
\begin{equation}\label{duality:one}
  \begin{split}
    \theta(\phi^{-1}(v))\phi^{-1}(v)\leq 
    & (-1-\theta(\phi^{-1}(v)))(\gamma(s_2)-\gamma(s_1)- \phi^{-1}(v))+\gamma(s_2)-\gamma(s_1)\\
    & +\xi(\vert \gamma(s_1)\vert)+\xi(\vert \gamma(s_2)\vert).
  \end{split}
\end{equation}

Third, introduce 
\begin{align*}
W^i & :=\xi(\vert S_i\vert)-\mu_i[\xi(\vert \cdot \vert)]-1, && \quad i\in \{1,2\},\\
W^{i+2} & :=\xi(\vert \gamma(S_i)\vert)-\mu_i[\xi(\vert \gamma\vert)]-1, && \quad i \in \{1,2\},\\
W^5 & :=\theta(\phi^{-1}(V))\phi^{-1}(V)-m-1, && 
\end{align*}
where $$m:=\mu_2[\gamma]-\mu_1[\gamma]+\mu_1[\xi(\vert \gamma\vert)]+\mu_2[\xi(\vert \gamma\vert)]\in [0,\infty).$$ 
Let also \[G:=\{\alpha\cdot (W^1,\ldots,W^5):\alpha\in[0,\infty)^5\}\] and \[\Pi:=\{\text{probabilities $\pi$ on }\mathcal{B}_{\JJ^2\times [0,\infty)}\text{ with }\pi[W^i]\text{ well-defined, finite and }\leq 0\text{ for all } i\in \{1,\ldots, 5\}\}.\] Then the $W^i$, $i\in \{1,\ldots, 5\}$, are all strictly negative at $(s_0,s_0,0)$, where $s_0:=\mu_1[\mathrm{id}]=\mu_2[\mathrm{id}]\in\JJ$, and hence \cite[Lemma~4.5]{guyon} implies  that for all $\tilde{f}\in \mathcal{B}_{\JJ^2\times [0,\infty)}/\mathcal{B}_\mathbb{R}$ (for which $\pi[\tilde{f}]$ is well-defined for all $\pi\in\Pi$)
\begin{equation}\label{duality:two}
\inf\{x\in \mathbb{R}:x+g\geq \tilde{f}\text{ for some }g\in G\}=\sup_{\pi \in \Pi}\pi[\tilde{f}].
\end{equation} 

Now let $H$ be the set of functions $h\in \mathbb{R}^{\JJ^2\times [0,\infty)}$ of the form $h(s_1,s_2,v)=u_1(s_1)-\mu_1[u_1]+u_2(s_2)-\mu_2[u_2]+\Delta(s_1,v)(s_2-s_1)+\Gamma(s_1,v)(\gamma(s_2)-\gamma(s_1)-\phi^{-1}(v))$, where $(u_1,u_2,\Delta,\Gamma)\in \mathcal{L}^1(\mu_1)\times \mathcal{L}^1(\mu_2)\times b\mathcal{B}_{\JJ\times [0,\infty)}\times (\mathcal{B}_{\JJ\times [0,\infty)}/\mathcal{B}_\mathbb{R})$; $H^c$ and $H^{cb}$ are the subspaces of $H$ obtained when $(u_1,u_2,\Delta,\Gamma)$ are restricted to continuous and bounded continuous functions, respectively. Then $$D:= \inf_{(u_1,u_2,\Delta,\Gamma)\in \mathcal{U}'}\mu_1[u_1]+\mu_2[u_2]=\inf\{x\in \mathbb{R}:x+h\geq V\text{ for some }h\in H\}$$
by the definition of $\mathcal{U}'$ and $H$, and so
\begin{align*} D & \leq \inf\{x\in \mathbb{R}:x+h\geq V\text{ for some }h\in H^c\}\\
 & =\inf\{x\in \mathbb{R}:x+g+h\geq V\text{ for some }g\in G,h\in H^c\},
\end{align*}
because, by the arguments above, and \eqref{duality:one} in particular, for every $g\in G$ there exists an $h\in H^c$ with $g\leq h$, and since $H^c$ is a linear space and $0\in G$. Then
\begin{align*}
D & \leq \inf\{x\in \mathbb{R}:x+g+h\geq V\text{ for some }g\in G,h\in H^{cb}\}\\
& =\inf_{h\in H^{cb}}\inf\{x\in \mathbb{R}:x+g\geq V-h\text{ for some }g\in G\}\\
& =\inf_{h\in H^{cb}}\sup_{\pi \in \Pi}\pi[V-h] \qquad \qquad \qquad \text{(by \eqref{duality:two})}\\
& = \sup_{\pi \in \Pi}\inf_{h\in H^{cb}} \pi[V-h].
\end{align*}
The final step follows by the same line of reasoning as in \cite[p. 608]{guyon}, exploiting (i) the concavity of $\phi$, which yields that $V\leq A+B\phi^{-1}(V)$ for some $\{A,B\}\subset [0,\infty)$; (ii) the assumption $\lim_{(\inf\JJ) +}(\vert\gamma\vert+\vert \mathrm{id}\vert)=\lim_{(\sup\JJ) -}(\vert\gamma\vert+\vert \mathrm{id}\vert)=\infty$; and (iii) the bound $\vert V-h\vert \leq A(1+\phi^{-1}(V)+\vert S_1\vert+\vert S_2\vert+\vert \gamma(S_1)\vert+\vert \gamma(S_2)\vert)$ for some $A\in [0,\infty)$, in order to show that $\Pi$ is weakly compact and  that $\pi\mapsto \pi[V-h]$ is weakly upper semicontinuous for $h\in H^{cb}$ (\cite[Lemma~4.8]{guyon}). Finally, from the definitions of $\Pi$, $H^{cb}$ and $\mathcal{M}'$ we conclude that
$$D \leq  \sup_{\mu \in \MM'\cap \Pi}\mu[V]\leq \sup_{\mu \in \MM'}\mu[V].$$

An inspection of the above reveals that we would still be able to prove the inequality ``$\geq$'' in \eqref{duality} if in the definition of $\mathcal{U}'$ all of the functions were continuous and with each of $u_1,u_2$ being moreover the difference of two convex functions (a ``delta-convex'' function): the latter e.g. because of the results of \cite{hartman} that gives closure of the set of delta-convex functions under compositions, subject to conditions that are sufficiently innocuous to apply in the present case, i.e. to $\xi(\vert\gamma\vert)$ in \eqref{duality:one}.

Finally, the argument that the supremum in \eqref{duality} is attained is exactly the same as in \cite[p.~607, 1st paragraph of proof of Theorem~4.1]{guyon}, though we should point out that the inclusion $\mathcal{M}'\subset \Pi$ is not clear (and probably not true, not even in the setting of \cite{guyon}), however this is not important, because by the argument above $ \sup_{\mu \in \MM'}\mu[V]=\sup_{\mu \in \MM'\cap \Pi}\mu[V]$ and the set $\MM'\cap \Pi$ is weakly compact with $V\mapsto \mu[V]$ weakly upper semicontinuous. 
\end{proof}
The proof of the following result is a straightforward computation; it corresponds to the ``classical superreplicating portfolio'' of \cite[Eq.~(2.7)]{guyon}.

\begin{proposition}
Set $v^*:=\phi(\mu_2[\gamma]-\mu_1[\gamma])$. Define the Legendre transform $(\phi^{-1})^*:[0,\infty)\to [0,\infty]$ of $\phi^{-1}$ as follows: 
$$(\phi^{-1})^*(b):=\sup_{u\in [0,\infty)}(bu-\phi^{-1}(u)).$$
Assume the equation 
\begin{equation}\label{equation}
(\phi^{-1})^*(b)=bv^*-\phi^{-1}(v^*),
\end{equation}
in $b\in (0,\infty)$, has a solution $b^*$. Then, setting $a^*:=1/b^*$, $u_1^*:=v^*-a^*\phi^{-1}(v^*)-a^*\gamma$, $u_2^*:=a^*\gamma$, $\Delta^*:=0$ and $\Gamma^*=-a^*$, we have

\hfill $\displaystyle(u_1^*,u_2^*,\Delta^*,\Gamma^*)\in \UU'\text{ and }\mu_1[u_1^*]+\mu_2[u_2^*]=v^*.$\qed
\end{proposition}
\begin{remark}
If \eqref{duality} and the equivalent conditions of Proposition~\ref{proposition:good_scenario} prevail, then $\mu_1[u_1^*]+\mu_2[u_2^*]=\min_{(u_1,u_2,\Delta,\Gamma)\in \mathcal{U}'}\mu_1[u_1]+\mu_2[u_2]$. 
\end{remark}
\begin{remark}
Eq.~ \eqref{equation} has  a solution when $\phi=\sqrt{\cdot}$ and $v^*>0$, namely $b^*=2v^*$.
\end{remark}

\bibliographystyle{plainnat}
\bibliography{Biblio_VIX}
\end{document}